\documentclass[reqno,centertags, 12pt, draft]{article}

\usepackage{amsmath,amsthm,amscd,amssymb}

\usepackage{latexsym}

\usepackage{mathrsfs}

\newcommand{\bbN}{{\mathbb{N}}}

\newcommand{\bbR}{{\mathbb{R}}}

\newcommand{\bbZ}{{\mathbb{Z}}}

\newcommand{\bbC}{{\mathbb{C}}}

\newcommand{\calD}{{\mathcal D}}

\newcommand{\ip}[2]{ \left \langle #1 , #2 \right \rangle }

\newcommand{\no}{\nonumber}

\newcommand{\ti}{\tilde  }

\newcommand{\wti}{\widetilde  }

\newcommand{\beq}{\begin{equation}}

\newcommand{\eeq}{\end{equation}}

\newcommand{\ba}{\begin{align}}

\newcommand{\ea}{\end{align}}

\def\Im{{\rm Im}\,}

\allowdisplaybreaks

\numberwithin{equation}{section}

\newtheorem{theorem}{Theorem}[section]

\newtheorem{proposition}[theorem]{Proposition}

\newtheorem{lemma}[theorem]{Lemma}

\newtheorem{corollary}[theorem]{Corollary}

\theoremstyle{definition}

\newtheorem{definition}[theorem]{Definition}

\theoremstyle{remark}

\newtheorem*{remark*}{Remark}

\title{Eigenvalue Spacings and Dynamical Upper Bounds for Discrete One-Dimensional Schr\"odinger Operators}
\author{Jonathan Breuer, Yoram Last, and Yosef Strauss \\
\footnotesize Institute of Mathematics,  The Hebrew University of Jerusalem, Givat Ram, \\
\footnotesize 91904 Jerusalem, Israel. \\
\footnotesize Email: jbreuer@math.huji.ac.il, ylast@math.huji.ac.il, ystrauss@cs.bgu.ac.il}
\date{October 26, 2009} 

\begin{document}
\sloppy
\maketitle
\begin{abstract}
We prove dynamical upper bounds for discrete one-dimensional Schr\"odinger operators in terms of 
various spacing properties of the eigenvalues of finite volume approximations. We demonstrate the 
applicability of our approach by a study of the Fibonacci Hamiltonian.
\end{abstract}


\section{Introduction}

Let $H=\Delta+V$ be a discrete, one-dimensional, bounded
Schr\"odinger operator on $\ell^2(\bbZ)$ or $\ell^2(\bbN)$: \beq
\label{schrodinger} \left( H\psi \right)_n=\psi_{n+1}+\psi_{n-1}+V_n
\psi_n , \eeq where $V_n$ is a bounded real valued function on $\bbZ$
or $\bbN$ and in the case of $\bbN$, $\left( H\psi
\right)_1=\psi_2+V_1 \psi_1$. We are interested here in the unitary
time evolution, $e^{-itH}$, generated by $H$.
A wave packet, $\psi(t) = e^{-itH}\psi$, which is
initially localized in space, tends to spread out in time. Connections between
the rate of this spreading and spectral properties of $H$ have
been the subject of extensive research in the last three decades.
However, while various lower bounds have been obtained in many
interesting cases using an assortment of different approaches, the
subject of upper bounds is significantly less well understood.

The purpose of this work is to formulate general dynamical upper
bounds in terms of purely spectral information. However, rather than
consider spectral measures of $H$, the infinite volume object, we
shall consider finite volume approximations to $H$.  Among the
properties we consider, a central role will be played by the spacing
of eigenvalues of these finite volume approximations. In particular,
we shall show that the rate of spreading of a wave packet can be
bounded from above by the strength of eigenvalue clustering on
finite scales.

The spreading rate of a wave packet may be measured in various
different ways (for a survey of several of these see \cite{L}).
As is often done in this line of research, we shall
focus in this work on the average portion of the tail of the wave
packet that is outside a box of size $q$ after time $T$, \beq \no
P_{\psi}(q,T)=\sum_{|n|>q} \frac{2}{T}\int_0^\infty \left|
\ip{\delta_n}{\psi(t)} \right|^2 e^{-2t/T}dt , \eeq where
$\psi(t)=e^{-itH}\psi$ and $\ip{\cdot}{\cdot}$ is the inner
product. Moreover, as $\psi =\delta_1$ is an ideal candidate for a
wave packet that is localized at the origin, we shall further
restrict our attention to that case and denote \beq \no
P(q,T)\equiv P_{\delta_1}(q,T)= \sum_{|n|>q}\frac{2}{T}\int_0^\infty
\left| \ip{\delta_n}{e^{-itH}\delta_1} \right|^2 e^{-2t/T}dt. \eeq

Another important measure of the spreading rate of $\psi(t)$ is the
rate of growth of moments of the position operator: \beq \no \left
\langle \left\langle \left|X\right|^m \right \rangle \right
\rangle_T = \sum_{n} \frac{2}{T}\int_0^\infty  |n|^m \left|
\ip{\delta_n}{\psi(t)} \right|^2 e^{-2t/T} dt ,\eeq where $m > 0$.

These quantities are related. In particular, since for any
bounded $V$ the ballistic upper bound, $\left \langle \left\langle
\left|X\right|^m \right \rangle \right \rangle_T \leq C_m T^m$,
holds, it follows from known results (see \cite{GKT}) that if $P(T^\alpha,
T)=O(T^{-k})$ as $T \rightarrow \infty$ for all $k$, then $ \left
\langle \left\langle \left|X\right|^m \right \rangle \right
\rangle_T \leq C_m T^{\alpha m}$ for all $m$.

As mentioned above, various \emph{lower bounds} for $P_\psi(q,T)$
and its asymptotics, using local properties of $\mu_\psi$, the
spectral measure of $\psi$, have been obtained. A basic result in
this area which is often called the Guarneri-Combes-Last
bound \cite{comb, gu, gu2, L} says that if
$\mu_\psi$ is not singular with respect to the $\alpha$-dimensional
Hausdorff measure, then $P_\psi(q,T)$, for $q \sim T^\alpha$,
is bounded away from zero for all $T$. Various bounds on other related properties of $\mu_\psi$ 
have also been shown to imply lower bounds on transport (see, e.g., \cite{BGT, gs, GSB}). 
Roughly speaking, a central idea in all these works says that a ``higher degree of continuity'' of the
spectral measure implies faster transport.

Another idea, not unrelated to the one above, is to use growth
properties of generalized eigenfunctions to get dynamical
bounds. There have been many works in this direction (e.g., \cite{DT,DT2,KKL,KL}), 
some of which also combine generalized eigenfunction properties with properties of 
the spectral measure mentioned above.
Especially relevant to this paper is \cite{DT} which shows that it is sufficient to
have ``nice'' behavior of generalized eigenfunctions at a single
energy in order to get lower bounds on the dynamics. 
The relevance of this result here is to Theorem
\ref{th-1.8} below, which assumes control of \emph{all}
eigenvalues of $H^q(\pi/2)$. The result of \cite{DT} quoted
above may provide a clue as to whether this restriction is necessary
in a certain sense or simply a side-effect of our proof.

While there are fairly many results concerning \emph{lower bounds} on $P_\psi(q,T)$,
there seem to be much fewer works concerning
general \emph{upper bounds} on it. If $\mu_\psi$ is a pure
point measure, it follows from a variant of Wiener's theorem \cite[Theorem XI.114]{reed-simon3}
that $\lim_{q\rightarrow \infty} \lim_{T \rightarrow
\infty}P_\psi(q,T)=0$, showing that the bulk of the wave packet cannot
spread to infinity. For the Anderson model (where $V_n$ is a
sequence of i.i.d.\ random variables), the tails of the wave packet have been shown to
remain exponentially small for all times (see, e.g., \cite{KS}), implying boundedness of
$\left \langle \left\langle \left|X\right|^2 \right \rangle \right
\rangle_T$. However, as shown in \cite{DJLS}, near-ballistic growth of
$\left \langle \left\langle \left|X\right|^2 \right \rangle \right\rangle_T$
is also possible for pure point spectrum. 
If $\mu_\psi$ is continuous, there are examples
\cite{B,KL} showing that one may have fast transport
(e.g., near-ballistic spreading of the bulk of the wave packet) even
for cases where $\mu_\psi$ is very singular. It is thus understood
that, for continuous measures, there can be no meaningful upper bounds
in terms of continuity properties of the spectral measure alone.
Moreover, as bounding the growth rate of quantities like
$\left \langle \left\langle \left|X\right|^2 \right \rangle \right\rangle_T$
from above involves control of the full wave packet
(while it is sufficient to control only a portion of the wave packet
to bound such growth rates from below), obtaining upper bounds on
$P(q,T)$ that would be ``good enough'' to yield meaningful bounds
on the growth rates of moments of the position operator
appears to be a significantly more difficult problem then obtaining
corresponding lower bounds.

The few cases where upper bounds have been obtained for singular
continuous measures include a work of Guarneri and Schulz-Baldes \cite{GS2},
who consider certain Jacobi matrices (discrete Schr\"odinger
operators with non constant hopping terms) with self-similar spectra
and formulate upper bounds in terms of parameters of some related
dynamical systems, a work by Killip, Kiselev and Last \cite{KKL},
who obtain a fairly general upper bound on the spreading rate
of some portion of the wave packet, and a recent work of Damanik and Tcheremchantsev
\cite{DT3} (to which we shall return later on), who obtain dynamical upper
bounds from properties of transfer matrices. Out of these, the last
mentioned work is the only one obtaining tight control
over the entire wave packet, thus providing
upper bounds on $P(q,T)$ that are ``good enough'' to yield meaningful
bounds on the growth rates of moments of the position operator.

As mentioned above, our aim in this paper is to derive general
dynamical upper bounds in terms of purely spectral information.
Moreover, our bounds achieve tight control of the entire wave packet,
thus yielding meaningful bounds on the growth rates of moments of
the position operator (like those of \cite{DT3}).
It is clear, from the examples cited above, that local
properties of the spectral measure alone are not sufficient for this
purpose. Thus, we shall consider the spectral properties of finite
volume approximations to $H$.

Before we describe our results, let us introduce some useful notions.
It has become customary to use certain exponents to measure the rates of growth of various parts of 
the wave packet. Following \cite{GKT} (also see \cite{DT3}), we define
\beq \label{eq-1.2}
\begin{split}
\alpha_l^+ &= \sup \left\{\alpha > 0 \mid \limsup_{T \rightarrow \infty} \frac{\log P(T^\alpha, T)}{\log T}=0 \right\} \\
\alpha_l^- &= \sup \left \{\alpha > 0 \mid \liminf_{T \rightarrow \infty} \frac{\log P(T^\alpha, T)}{\log T}=0 \right \} 
\end{split}
\eeq
and 
\beq \label{eq-1.3}
\begin{split}
\alpha_u^+ &= \sup \left\{\alpha > 0 \mid \limsup_{T \rightarrow \infty} \frac{\log P(T^\alpha, T)}{\log T}>-\infty  \right \} \\
\alpha_u^- &= \sup \left\{\alpha > 0 \mid \liminf_{T \rightarrow \infty} \frac{\log P(T^\alpha, T)}{\log T}>-\infty \right \}. 
\end{split}
\eeq
$\alpha_l^\pm$ are interpreted as the rates of propagation of the `slow' moving part of the wave packet, while 
$\alpha_u^\pm$ are the rates for the `fast' moving part. The following notion is useful for applications.
\begin{definition} \label{def-1.1}
We call a sequence of nonnegative numbers, $\{a_n\}_{n=1}^\infty$, \emph{exponentially growing} if 
$\sup_n \frac{a_{n+1}}{a_n}<\infty$ and $\inf_n \frac{a_{n+1}}{a_n}>1$. 
\end{definition}

We are finally ready to describe our results. We first discuss the results for $H$ on $\bbN$, since the whole
line case will be reduced to this case. For $q>1, \ q \in \bbN$, $k
\in [0,\pi]$, let $H^q(k)$ be the restriction of $H$ to
$\{1,\ldots,q \}$ with boundary conditions
$\psi(q+1)=e^{ik}\psi(1)$, namely, \beq \label{Truncation}
H^{q}(k)=\left(
\begin{array}{ccccc}
V_1    & 1 & 0      & \ldots      & e^{-ik} \\
1    & V_2 & 1    & \ddots     & 0 \\
0      & \ddots & \ddots    & \ddots    & \vdots \\
\vdots     & \ddots & \ddots    &\ddots     &1\\
e^{ik} & 0   & \ldots & 1 & V_q \\
\end{array} \right).
\eeq As is well known (see, e.g., \cite[Section 7]{teschl}), for any $k \in (0,\pi)$,
$H^{q}(k)$ has $q$ simple eigenvalues, $E_{q,1}(k)<E_{q,2}(k)<
\ldots<E_{q,q}(k)$. $E_{q,j}(k)$ are continuous monotone functions
of $k$ and, as $k$ varies over $[0,\pi]$, they trace out bands. An
idea that goes back to Edwards and Thouless \cite{ET} (also see \cite{T}) 
is to use the width of the
bands, $b_{q,j} \equiv |E_{q,j}(\pi)-E_{q,j}(0)|$
(which can be identified as a measure of what is often called
``Thouless energy'' in physics literature), 
as a measure of the system's sensitivity to a variation of boundary conditions.
Faster spreading of the wave packet is intuitively associated with a greater degree of
extendedness of the eigenstates of $H$ and thus with a greater
sensitivity to a change in $k$. Our first theorem is motivated by this
intuitive picture:

\begin{theorem}\label{Thouless-width}
For $1 \leq j \leq q$, let $b_{q,j} \equiv
|E_{q,j}(\pi)-E_{q,j}(0)|$. Then
$$P(q,T) \leq \frac{4e^2}{\left(\sqrt{5}+1 \right)^2} \left(1+2 \parallel V \parallel_\infty \right)^2 T^6 
(\sup_{1 \leq j \leq q} b_{q,j})^2.$$
\end{theorem}

\begin{corollary} \label{cor-1.3}
Assume that there exist $\beta>3$ and a sequence, $\{q_\ell\}_{\ell=1}^\infty$, such that 
$\sup_{1 \leq j \leq q_\ell} b_{q_\ell,j} < q_\ell^{-\beta}$. Then $\alpha_l^- \leq 3/\beta$. 
If, moreover, the sequence $\{q_\ell\}_{\ell=1}^\infty$ above is exponentially growing, then $\alpha_l^+ \leq 3/\beta$.
\end{corollary}

The bands traced by $E_{q,j}(k)$ make up the spectrum of the whole
line operator with $q$-periodic potential, $V^{\rm{per}}$, defined
by $V^{\rm{per}}_{nq+j}=V_j$ ($1 \leq j \leq q$). 
An elementary argument shows that \cite[Theorem 1.4]{deift-simon} implies that 
$$
\sup_{1 \leq j \leq q}b_{q,j} \leq \frac{2\pi}{q} .
$$
Of course, we need stronger decay to get a meaningful bound.

The $\left(1+2
\parallel V \parallel_\infty \right)^2 T^6$ factor in Theorem~\ref{Thouless-width}
is partly due to the effective
approximation of $V$ by $V^{\rm{per}}$, which underlies our analysis. It can be replaced
by $T^2$ if $V$ happens to be periodic of period $q$ and it can be improved,
for some $q$'s, in cases where $V$ has good repetition properties that make it
close to being periodic for some scales.
As seen in Corollary \ref{cor-1.3}, this $O(T^6)$ factor implies that for
Theorem~\ref{Thouless-width} to be meaningful, $(\sup_j b_{q,j})^2$
(the squared maximal ``Thouless width'') must decay quite fast.
By considering additional information, one can do better. Let
$\widetilde{E}_{q,j}=E_{q,j}(\frac{\pi}{2})$. We will show that
control of the clustering properties of the $\wti{E}_{q,j}$ leads to
exponential bounds on $P(q,T)$.

\begin{definition} \label{clustering-definition}
We say that the set $\mathfrak{E}_q\equiv\{\widetilde{E}_{q,j}\}_{j=1}^q$
is {\it $(\varepsilon,\xi)$-clustered} if there exists a finite
collection $\{I_j\}_{j=1}^k$ $(k \leq q)$ of disjoint closed
intervals, each of size at most $\varepsilon$, such that
$\mathfrak{E} \subseteq \cup_{j=1}^k I_j$ and such that every $I_j$
contains at least $q^\xi$ points of the set $\{
\ti{E}_{q,j}\}_{j=1}^q$. When we want to be explicit about the cover, we say 
that $\mathfrak{E}_q$ is $(\varepsilon,\xi)$-clustered by $\{I_j\}_{j=1}^k$.
\end{definition}

\begin{theorem}\label{clustering}
Let $b_{q,j} \equiv
|E_{q,j}(\pi)-E_{q,j}(0)|$ for $1 \leq j \leq q$. For any $2/3< \xi
\leq 1$ and $0< \alpha <1$, there exist constants $\delta>0$ and
$q_0>0$ such that, if $q \geq q_0$ and the set $\mathfrak{E}_q \equiv \{ \widetilde{E}_{q,1},
\widetilde{E}_{q,2},\ldots,\widetilde{E}_{q,q} \}$ is
$(q^{-1/\alpha},\xi)$-clustered, then for $T \leq q^{1/\alpha}$, \beq
\label{cluster-bound} P(q,T) \leq 4e^2
\left(1+2\parallel V\parallel_\infty \right)^2 T^4 (\sup_{1\leq j \leq
q} b_{q,j})^2 e^{-C q^\delta}, \eeq where $C$ is some
universal constant. In particular, \beq \label{cluster-bound1}
P(q,q^{1/\alpha}) \leq 4e^2(1+2
\parallel V
\parallel_\infty)^2 q^{4/\alpha} (\sup_{1 \leq j \leq q} b_{q,j})^2 e^{-C q^{\delta}}.\eeq
\end{theorem}

\begin{corollary}\label{cor-1.6}
Assume that there exist $2/3< \xi
\leq 1$, $0< \alpha <1$ and a sequence $\{q_\ell\}_{\ell=1}^\infty$, such that the set $\mathfrak{E}_{q_\ell}$ is
$(q_\ell^{-1/\alpha},\xi)$-clustered for all $\ell\in \bbN$. Then $\alpha_u^-\leq \alpha$. 
If, moreover, the sequence $\{q_\ell\}_{\ell=1}^\infty$ above is exponentially growing, then $\alpha_u^+ \leq \alpha$.  
\end{corollary}

Thus, we see that the spreading rate of a wave packet can be bounded
by the strength of eigenvalue clustering at appropriate length scales. The final
ingredient in our analysis comes from considering different length
scales at once. The local bound of Theorem \ref{clustering} can be
improved if the clusters at different scales form a structure with a
certain degree of self-similarity. We first need some definitions.
As above, for any $q \in \bbN$, we let
$\mathfrak{E}_{q}=\{\widetilde{E}_{q,j}\}_{j=1}^{q}$.

\begin{definition} \label{def-1.4}
We say that the sequence $\{\mathfrak{E}_{q_\ell}\}_{\ell=1}^\infty$ is
\emph{uniformly clustered} if $\mathfrak{E}_{q_\ell}$ is
$\{q_\ell^{-1/\alpha_\ell}, \xi_\ell \}$-clustered by
$U_\ell=\{I_j^\ell \}_{j=1}^{k_\ell}$ and the following hold:
\\
(i) If $\ell_1<\ell_2$ then $q_{\ell_1}^{-1/\alpha_{\ell_1}}>q_{\ell_2}^{-1/\alpha_{\ell_2}}$. \\
(ii) There exist $\mu \geq 1$ and a constant $C_1>0$ so that \beq
\no \inf_{1 \leq j \leq k_\ell} |I^\ell_j| \geq C_1
q_\ell^{-\mu/\alpha_\ell}.
\eeq \\
(iii) There exists a $\delta>0$ so that $ \delta < \xi_\ell
<(1-\delta)$
and $\delta< \alpha_\ell<1$ for all $\ell$. \\
(iv) Define ${\bar{\xi}_\ell} = \log\left(\sup_{1 \leq j \leq k_\ell}
\# \left( \mathfrak{E}_\ell\cap I_j^\ell \right)\right)/\log q_\ell $,
then there exists a constant $C_2$ so that \beq
\no \bar{\xi}_\ell- \xi_\ell \leq
\frac{C_2}{\log q_\ell}.\eeq \\
When we want to be explicit about the cover and the relevant
exponents, we say that
$\{\mathfrak{E}_{q_\ell}\}_{\ell=1}^\infty$ is uniformly clustered
by the sequence $\{U_\ell\}_{\ell=1}^\infty$ ($U_\ell=\{I_j^\ell\}_{j=1}^{k_\ell}$) with exponents
$\{\alpha_\ell, \xi_\ell, \mu \}$.
\end{definition}

\begin{definition} \label{def-1.5}
Let $\left\{ U_\ell=\{I_j^\ell\}_{j=1}^{k_\ell}
\right\}_{\ell=1}^\infty$ be a sequence of sets of intervals such
that $\varepsilon_\ell \geq |I_j^\ell| \geq C \varepsilon_\ell^\mu$ for a
monotonically decreasing sequence
$\{\varepsilon_\ell\}_{\ell=1}^\infty$ and constants $C>0$ and
$\mu \geq 1$. Let $0<\omega<1$. We say that
$\{U_\ell\}_{\ell=1}^\infty$ \emph{scales nicely with exponents
$\mu$ and $\omega$} if for any $1>\varepsilon>0$ there exists a set
of intervals,
$U_\varepsilon=\{I_j^\varepsilon\}_{j=1}^{k_\varepsilon}$, of length
at most $\varepsilon$ and no less than $C \varepsilon^\mu$, such
that $U_{\varepsilon_\ell}=U_\ell$ with the following properties: \\
(i) If $\varepsilon_1> \varepsilon_2$ then for any $1\leq j \leq
k_{\varepsilon_2}$ there exists an  $1\leq m \leq k_{\varepsilon_1}$
such that $I_j^{\varepsilon_2} \subseteq I_m^{\varepsilon_1}$. \\
(ii) There exists a constant $C_3>0$ such that if $\varepsilon_1>
\varepsilon_2$ then for any $1\leq m \leq k_{\varepsilon_1}$, $\#\{j
\mid I_j^{\varepsilon_2} \cap I_m^{\varepsilon_1} \neq \emptyset \}
\leq C_3 \left( \varepsilon_1/\varepsilon_2 \right)^\omega$.
\end{definition}

In simple words, $\{U_\ell\}_{\ell=1}^\infty$ scales nicely if the sequence may be extended to a `continuous' family 
(parameterized by the interval lengths) in such a way that intervals of one length scale are contained in and, 
to a certain extent, `nicely distributed' among the intervals of a larger length scale.

\begin{remark*}
We emphasize that, while the families $U_\ell$ in Definition \ref{def-1.4} are assumed to consist 
of \emph{disjoint} intervals, we do not assume this disjointness about the families $U_\varepsilon$ in Definition 
\ref{def-1.5}.
\end{remark*}

\begin{remark*}
An example of a nicely scaling sequence is given by the sequence
$\{U_l\}_{l=1}^\infty$ where $U_1= \left \{[0,1/3], [2/3,1]
\right\}$, $U_2=\left \{[0,1/9],\ [2/9,1/3],\ [2/3, 7/9],\ [8/9,1]
\right \}$ and so on ($U_l$ is the set of intervals obtained by
removing the middle thirds of the intervals comprising $U_{l-1}$).
It is not hard to see that this sequence scales nicely with
exponents $\mu=1$ and $\omega=\frac{\log 2}{\log 3}$ (setting $2$ as the value of $C_3$ in Definition \ref{def-1.5}).
\end{remark*}

\begin{theorem} \label{th-1.8}
Assume that $\{q_\ell\}_{\ell=1}^\infty$ is a sequence such that
$\{\mathfrak{E}_{q_\ell}\}_{\ell=1}^\infty$ is uniformly clustered
by $\left \{U_\ell=\{I_j^\ell\}_{j=1}^{k_\ell} \right \}_{\ell=1}^\infty$ with exponents
$\{\alpha_\ell, \xi_\ell, \mu \}$. Suppose, moreover, that
$\{U_\ell\}_{\ell=1}^\infty$ scales nicely with exponents $\mu$ and
$\omega$ for some $0<\omega<1$. Assume also that, for some $\zeta>0$,
\beq \label{eq-1.5} 2\omega \left(\frac{\mu-1}{\mu-\omega}
\right)+\zeta<\xi_\ell \alpha_\ell. \eeq

Then for any $m > 0$ there exists $C_m>0$ such that 
\beq \label{eq-1.6} 
P(q_\ell,T)\leq C_m q_\ell^{-m}
\eeq for any $\ell \in \bbN$ and $T \leq q_\ell^{1/\alpha_\ell}$.
\end{theorem}

\begin{corollary}\label{cor-1.10}
Under the assumptions of Theorem \ref{th-1.8}, $\alpha_u^-\leq \liminf_{\ell \rightarrow \infty} \alpha_\ell$. 
If, moreover, the sequence $\{q_\ell\}_{\ell=1}^\infty$ in the theorem is exponentially growing, then 
$\alpha_u^+ \leq \limsup_{\ell \rightarrow \infty} \alpha_\ell$. 
\end{corollary}

Note that \eqref{eq-1.5} says that one may trade strong clustering
for greater degree of uniformity over different length scales: When
$\mu=1$ (so cluster sizes are very uniform), \eqref{eq-1.5} says
that the `clustering strength' parameter---$\xi_\ell$---only has to
be positive. On the other hand, when ``$\mu=\infty$'' \eqref{eq-1.5}
says $\xi_\ell\alpha_\ell>2\omega$. Now note that, since there are at
least $q^\xi$ eigenvalues in each interval, there are at most
$q^{1-\xi_\ell}=\varepsilon^{-\alpha_\ell(1-\xi_\ell)}$ intervals, which shows that the assumption 
$\omega \sim \alpha_\ell(1-\xi_\ell)$ is a natural one. But this, combined with $\xi_\ell\alpha_\ell>2\omega$, 
immediately implies $\xi_\ell>2/3$,
which is the condition of Theorem \ref{clustering}. 

Now let $H=\Delta+V$ be a
\emph{full line} Schr\"odinger operator. 
We shall treat $H$ by reducing the analysis to that of two corresponding half-line cases. 
Let \beq \no
H^{\pm}=\left(
\begin{array}{ccccc}
V_{\pm1}    & 1 & 0      & \ldots  & \ldots     \\
1    & V_{\pm2} & 1    & \ddots    & \ddots  \\
0      & 1 & \ddots    & \ddots    &\ddots \\
\vdots   &\ddots  & \ddots & V_{\pm n}    &\ddots     \\
 \vdots   & \ddots & \ddots & \ddots & \ddots \\
\end{array} \right).
\eeq be the restrictions of $H$ to the positive and negative
half-lines, with the restriction to the negative half-line rotated
to act on $\ell^2(\bbN)$ for notational convenience. (Note that
$V_0$ is not present in either $H^+$ or $H^-$). Furthermore, let
\beq \no P_{\delta_0}(q,T)= \sum_{|n|>q}\frac{2}{T}\int_0^\infty
\left| \ip{\delta_n}{e^{-itH}\delta_0} \right|^2 e^{-2t/T}dt \eeq
and \beq \no  P^{\pm}_{\delta_1}(q,T)=
\sum_{n>q}\frac{2}{T}\int_0^\infty \left|
\ip{\delta_n}{e^{-itH^{\pm}}\delta_1} \right|^2 e^{-2t/T}dt, \eeq so
that the sum in the second formula is restricted to $\bbN$. Then

\begin{proposition}\label{full-line}
For any $q>1$ and any $T>0$, \beq \label{full-line-from-half}
P_{\delta_0}(q,T) \leq T^2 \left(
P^{+}_{\delta_1}(q,T)+P^{-}_{\delta_1}(q,T) \right) \eeq
\end{proposition}

Thus, as remarked above, the full-line problem may be reduced to two
half-line problems (albeit with an extra factor of $T^2$).
Accordingly, Theorems \ref{Thouless-width}, \ref{clustering} and \ref{th-1.8} above 
imply corresponding theorems and corollaries similar to 
Corollaries \ref{cor-1.3}, \ref{cor-1.6} and \ref{cor-1.10} 
for a full line operator. Formulating these
results is straightforward, so we leave that to the interested
reader. We only note that, in the case of the application of Theorems \ref{clustering} and \ref{th-1.8} and the corresponding 
corollaries, the $T^2$ factor is of minor significance because of the existence of exponential bounds. 
In the application of Theorem \ref{Thouless-width}, however, this factor is clearly significant. 

To demonstrate the applicability of our results we use them to
obtain a dynamical upper bound for the Fibonacci Hamiltonian,
which is the most studied one-dimensional model of a quasicrystal.
This is the operator with $V$ given by: \beq
\label{fibonacci} V^\lambda_{\rm{Fib};n}=\lambda \chi_{[1-\theta,1)}(n\theta \
\rm{mod} 1), \quad \theta=\frac{\sqrt{5}-1}{2} \eeq with a coupling
constant $\lambda>0$ and where $\chi_I$ is the characteristic
function of $I$. For a review of some of the properties of the Fibonacci Hamiltonian, see \cite{d3}.

A subballistic upper bound for the fast-spreading part of the wave packet under the
dynamics generated by the Fibonacci Hamiltonian has been 
recently obtained by Damanik and Tcheremchantsev in \cite{DT3, DT4}. 
They used lower bounds on the growth of transfer matrices off the real line to 
obtain an upper bound on $\alpha_u^+$. In particular, they confirm the asymptotic dependence 
$\alpha_u^+ \sim \frac{1}{\log \lambda}$ (as $\lambda \rightarrow \infty$) of the transport exponent on the 
coupling constant as predicted by numerical calculations (see \cite{ah,ah1}). 
In Section 5 below, using purely spectral data for the Fibonacci Hamiltonian, we apply our above results to get the same 
asymptotics for $\alpha_u^+ (\lambda)$ (albeit with worse constants than those of \cite{DT3}).

We note that our primary purpose in this paper is to establish general
bounds which are based on clustering and self-similar multiscale clustering
of appropriate eigenvalues. In particular, we aim to highlight the connection
between self-similar Cantor-type spectra and what is often called ``anomalous transport.''
Roughly speaking, our results show that as long as the self-similar Cantor-type
structure manifests itself in a corresponding tight behavior of the eigenvalues
of appropriate finite-volume approximations of the operator, meaningful upper
bounds can indeed be obtained (and combined with existing lower bounds to
establish the occurrence of anomalous transport). We further note that
applying our general results to the Fibonacci Hamiltonian is
done mainly to demonstrate their applicability in their present form. In cases
where one is interested in obtaining bounds for concrete models, it is likely
that by using some of our core technical ideas below along with the most
detailed relevant data available for the concrete model in question one will be
able to establish stronger bounds than those obtained by direct application
of our above results.

The rest of this paper is structured as follows. Section 2 has some
preliminary estimates that will be used throughout the paper. The
proof of Theorem \ref{Thouless-width} is also given there. Section 3
has the proof of Theorems \ref{clustering} and \ref{th-1.8}. Section 4 has the proofs of corollaries \ref{cor-1.3}, 
\ref{cor-1.6} and \ref{cor-1.10} and of Proposition \ref{full-line}. Section 
5 describes the application of our results to the Fibonacci Hamiltonian.

\emph{Acknowledgment}. This research was supported
by The Israel Science Foundation (Grant No.\ 1169/06).


\section{Preliminary Estimates and Proof of Theorem \ref{Thouless-width}}

Let $H=\Delta+V$ be a discrete Schr\"odinger operator on $\bbN$ with
$V$, the potential, a bounded real-valued function. For $q \in
\bbN$, $k \in [0,\pi]$, let $H^q(k)$ be defined as in
\eqref{Truncation}. Finally, let
$H^q_{\textrm{per}}=\Delta+V^q_{\textrm{per}}$ where
$V^q_{\textrm{per}}(j+nq)=V(j)$ for $1\leq j \leq q$ and $n \geq 0$.

We start by deriving an inequality that will be central to
all subsequent developments. Let $G(k,n;z)=\ip{\delta_n}{\left(
H-z\right)^{-1} \delta_k}$ and $G^q_{\textrm{per}}(k,n;z)=
\ip{\delta_n}{\left( H^q_{\textrm{per}}-z\right)^{-1} \delta_k}$. We
define \beq \label{discriminant}
\calD^q(z)=G^q_{\textrm{per}}(1,q;z)+\frac{1}{G^q_{\textrm{per}}(1,q;z)}.
\eeq It is not hard to see, by recognizing $\calD^q(z)$ as the trace
of a transfer matrix, that it is a \emph{monic} polynomial of degree
$q$. This polynomial is called the discriminant for
$H^q_{\textrm{per}}$. It has been studied extensively in connection
with the spectral analysis of periodic Schr\"odinger operators (see
e.g.\ \cite[Section 7]{teschl}---the discriminant there is half ours). 
Among its properties that will be useful for us are:
\begin{enumerate}
\item The zeros of $\calD^q$ are precisely the $q$ distinct eigenvalues of
$H^q(\frac{\pi}{2})$, i.e., the set $\{ \wti{E}_{q,1},
\wti{E}_{q,2},\ldots,\wti{E}_{q,q} \}$ (and so are real and simple).
\item The essential spectrum of
$H^q_{\textrm{per}}$ is given by the inverse image under $\calD^q$
of the set $[-2,2]$.
\item The restriction of $\calD^q$ to $\bbR$ takes
the values $2$ and $-2$ precisely at the eigenvalues of $H^q(0)$ and
of $H^q(\pi)$, and for  any $1 \leq j \leq q$, $\left|
\calD^q\left(E_{q,j}(\pi) \right)-\calD^q\left(E_{q,j}(0) \right)
\right|=4$.
\item $(\calD^q)'(E)=0$ implies $E \in \bbR$ and
$|\calD^q(E)| \geq 2$.
\end{enumerate}
Our starting point is

\begin{lemma} \label{first-lemma}
For $q>1$ and for any positive $T$, \beq \label{central-inequality}
P(q,T) \equiv P_{\delta_1}(q,T) \leq 4 T^4 \left(1+2 \parallel V
\parallel_\infty \right)^2 \left( \inf_{E \in \bbR} \left | \calD^q(E+i/T)\right | \right)^{-2}\eeq
\end{lemma}

\begin{proof}

We start by recalling the formula \cite{KKL}: \beq
\label{Parseval} \int_0^\infty \left|
\ip{\delta_n}{e^{-itH}\delta_k} \right|^2 e^{-2t/T}dt=\frac{1}{2
\pi} \int_{-\infty}^\infty \left|\ip{\delta_n}{\left(H-E-i/T
\right)^{-1} \delta_k} \right|^2dE \eeq which is key in many recent
works on this topic (this work included).

By \eqref{Parseval} we see that \beq \label{starting-point}
\begin{split} P(q,T) &=\frac{1}{ \pi} \sum_{n>q}  \frac{1}{T}
\int_{-\infty}^\infty \left|\ip{\delta_n}{\left(H-E-i/T \right)^{-1}
\delta_1} \right|^2dE \\ &=\frac{1}{\pi} \int_{\bbR}dE \sum_{n>q}
\varepsilon \left |G(1,n; E+i\varepsilon) \right|^2,
\end{split} \eeq
where we use $\varepsilon=\frac{1}{T}$.

Now, let $\wti{H}^q=H-\ip{\delta_{q+1}}{\cdot} \delta_q -\ip{\delta_q}{\cdot}\delta_{q+1}$ and let
$\wti{G}^q(k,n;z)=\ip{\delta_n}{\left( \ti{H}^q-z\right)^{-1}
\delta_k}$. Then, by the resolvent formula \beq \label{resolvent1}
\begin{split}
G(1,n;z)&=\wti{G}^q(1,n;z)-G(1,q;z)\wti{G}^q(q+1,n;z)-G(1,q+1;z)\wti{G}^q(q,n;z)
\\ &=-G(1,q;z)\wti{G}^q(q+1,n;z),\end{split}\eeq
if $n>q$ since $\wti{G}$ is a direct sum. Moreover, note that
$\varepsilon\sum_{n>q}|\wti{G}^q(q+1,n;E+i\varepsilon)|^2=\Im
\wti{G}^q(q+1,q+1;E+i\varepsilon),$ (this can be seen by noting that 
$f(n) \equiv \wti{G}^q(q+1,n;E+i\varepsilon)$, satisfies $f(n+1)+f(n-1)+V(n)f(n)=(E+i\varepsilon)f(n)$ 
for all $n>q$; now multiply this by $\overline{f(n)}$, sum up and take imaginary parts).
Thus, we get from \eqref{starting-point} that \beq
\label{transport-Green} P(q,T)=\frac{1}{\pi} \int_{\bbR} \left
|G(1,q; E+i\varepsilon) \right|^2
\Im\wti{G}^q(q+1,q+1;E+i\varepsilon) dE, \eeq which, by
$\frac{1}{\pi} \int_{\bbR} \Im\wti{G}^q(q+1,q+1;E+i\varepsilon)
dE=1$, implies immediately that \beq
\label{transport-Green2}P(q,T)\leq \sup_{E \in \bbR}
|G(1,q,E+i\varepsilon)|^2. \eeq

Next, we approximate $G$ by $G^q_{\textrm{per}}$. Let
$\wti{H}_{\textrm{per}}^q=H^q_{\textrm{per}}-\ip{\delta_{q+1}}{\cdot}\delta_q -\ip{\delta_q}{\cdot}\delta_{q+1}$ and let
$\wti{G}_{\textrm{per}}^q(k,n;z)=\ip{\delta_n}{\left(
\wti{H}^q_{\textrm{per}}-z\right)^{-1} \delta_k}$. Note that
\eqref{resolvent1} holds with $G$ replaced by $G^q_{\textrm{per}}$
and $\wti{G}^q$ replaced by $\wti{G}^q_{\textrm{per}}$. Again, by
the resolvent formula (with $z=E+i \varepsilon$), \beq \label{G-Gper}
\begin{split} \left|G^q_{\textrm{per}}(1,q;z)-G(1,q;z) \right| &=
\left| \sum_{n
>q} G^q_{\textrm{per}}(1,n;z) \left(V^q_{\textrm{per}}(n)-V(n) \right) G(n,q;z)\right| \\
&= \left|G^q_{\textrm{per}}(1,q;z) \right| \\
& \times \left|\sum_{n>q}\wti{G}^q_{\textrm{per}}(q+1,n;z)
\left(V^q_{\textrm{per}}(n)-V(n) \right) G(n,q;z)  \right| \\
& \leq \left|G^q_{\textrm{per}}(1,q;z) \right| 2  \parallel V  \parallel_\infty \\
& \times \sqrt{\frac{\Im \wti{G}^q_{\textrm{per}(q+1,q+1;z)}}{\varepsilon}\frac{\Im G(q,q;z)}{\varepsilon}} \\
& \leq \frac{2 \parallel V \parallel_\infty}{\varepsilon^2} \left|G^q_{\textrm{per}}(1,q;z) \right|,
\end{split}\eeq
where the first inequality follows by applying Cauchy-Schwarz and the second inequality follows from
$|G(k,l;z)| \leq \frac{1}{\Im z}$.
This immediately implies
\beq \label{G-Gper2}
\begin{split}
|G(1,q;E+i\varepsilon)|&\leq
\left(1+\frac{2 \parallel V \parallel_\infty}{\varepsilon^2}\right) \left |G^q_{\textrm{per}}(1,q;E+i\varepsilon)\right|\\
&\leq \frac{1+2 \parallel V \parallel_\infty}{\varepsilon^2} \left |G^q_{\textrm{per}}(1,q;E+i\varepsilon)\right|.
\end{split}
\eeq

Now, since $G^q_{\textrm{per}}(1,n;z)$ is the exponentially decaying
solution to
$\psi(n+1)+\psi(n-1)+V^q_{\textrm{per}}(n)\psi(n)=z\psi(n)$ (for
$n>2$), it follows that $\left|G^q_{\textrm{per}}(1,q;z)\right|<1$
and so, by \eqref{discriminant}, that $2 \left|G^q_{\textrm{per}}
\right|^{-1} \geq |\calD^q(z)|$. This implies \beq \label{G-disc}
\left| G^q_{\textrm{per}}(1,q;z) \right| \leq
\frac{2}{|\calD^q(z)|}. \eeq

Combining \eqref{G-disc},\eqref{G-Gper2}, and \eqref{transport-Green2} (remembering that $\varepsilon=T^{-1}$) finishes
the proof.
\end{proof}

Thus, our problem is reduced to the analysis of $|\calD^q(z)|$. As
$\calD^q$ is monic we know \beq \label{Disc-form}
\calD^q(z)=\prod_{j=1}^q \left(z-\wti{E}_{q,j} \right). \eeq
Moreover, as the zeros of $(\calD^q)'$ are simple, any point $E \in
\bbR$ lies between two extremal points of $\calD^q$ or between an
extremal point and $\pm \infty$. Thus, each point $E \in \bbR$ lies
in an interval of the form $[x(E),y(E))$, or $[x(E), \infty)$, or
$(-\infty, y(E))$ where $x(E),y(E)$ are two extremal points of
$\calD^q$ and the interval contains no other extremal points. Any
such interval contains a unique zero of $\calD^q$. In this way we
associate with each point $E \in \bbR$ a unique zero
$\wti{E}_{q,j(E)}$ of $\calD^q$. Using this notation, for any $E \in
\bbR$ \beq \label{Dec-Discr}\begin{split}
|\calD^q(E+i\varepsilon)|^2 &=\prod_{j=1}^q
|E+i\varepsilon-\wti{E}_{q,j}|^2
\\
&=|E+i\varepsilon-\wti{E}_{q,j(E)}|^2\prod_{j \neq
j(E)}|E-\wti{E}_{q,j}|^2\frac{\prod_{j \neq
j(E)}|E+i\varepsilon-\wti{E}_{q,j}|^2}{\prod_{j \neq
j(E)}|E-\wti{E}_{q,j}|^2}  \\
& \geq \varepsilon^2 \prod_{j \neq j(E)}|E-\wti{E}_{q,j}|^2
\frac{\prod_{j \neq j(E)}|E+i\varepsilon-\wti{E}_{q,j}|^2}{\prod_{j
\neq
j(E)}|E-\wti{E}_{q,j}|^2}  \\
&=\varepsilon^2
\left|\frac{\calD^q(E)}{E-\wti{E}_{q,j(E)}}\right|^2\frac{\prod_{j
\neq j(E)}|E+i\varepsilon-\wti{E}_{q,j}|^2}{\prod_{j \neq
j(E)}|E-\wti{E}_{q,j}|^2}.
\end{split} \eeq

In a sense, all our theorems follow from lower bounds on the right hand side of
\eqref{Dec-Discr}. In particular, Theorem \ref{Thouless-width}
follows from noting $\frac{\prod_{j \neq
j(E)}|E+i\varepsilon-\wti{E}_{q,j}|^2}{\prod_{j \neq
j(E)}|E-\wti{E}_{q,j}|^2} \geq 1$ and studying the other terms,
while Theorems \ref{clustering} and \ref{th-1.8} follow from a more
detailed analysis of $\frac{\prod_{j \neq
j(E)}|E+i\varepsilon-\wti{E}_{q,j}|^2}{\prod_{j \neq
j(E)}|E-\wti{E}_{q,j}|^2}$.

We proceed now with the proof of Theorem \ref{Thouless-width}. We
first need a lemma.

\begin{lemma} \label{derivative-bound}
Recall $b_{q,j} \equiv |E_{q,j}(\pi)-E_{q,j}(0)|$. For any $E \in
[\wti{E}_{q,1}, \wti{E}_{q,q}]$ \beq \label{derivative-up-low-bound}
e \left|\frac{\calD^q(E)}{E-\wti{E}_{q,j(E)}} \right| \geq \left|
(\calD^q)'(\wti{E}_{q,j(E)})  \right| \geq \frac{\sqrt{5}+1}{b_{q,j}}
\eeq $($where, for $E=\wti{E}_{q,j}$ for some $j$, the left hand
side is interpreted as the derivative.$)$
\end{lemma}

\begin{remark*}
The proof of this lemma is essentially contained in the proof of
Lemma 1 of \cite{L3}, although it is stated somewhat
differently there (also see \cite[Theorem 5.4]{LS}). We
repeat it here (with some details omitted) for the reader's
convenience.
\end{remark*}

\begin{proof}
We start with the upper bound. For this it suffices to show that for
any $E \in [\wti{E}_{q,j}, \wti{E}_{q,j+1}]$ ($1 \leq j \leq (q-1)$)
\beq \label{derivative-upper-bound} e
\left|\frac{\calD^q(E)}{E-\wti{E}_{q,j}} \right| \geq \left|
(\calD^q)'(\wti{E}_{q,j(E)})  \right| \eeq

Enumerate the zeros of $(\calD^q)'$ by
$E^0_{q,1},\ldots,E^0_{q,q-1}$. Clearly,
$\wti{E}_{q,j}<E^0_{q,j}<\wti{E}_{q,j+1}$ for any $1 \leq j \leq
(q-1)$, and any $E \in [\wti{E}_{q,j}, \wti{E}_{q,j+1}]$ is
contained either in $[\wti{E}_{q,j}, E^0_{q,j}]$ or in
$[E^0_{q,j},\wti{E}_{q,j+1}]$. As the zeros of $(\calD^q)'$ are
simple, $E^0_{q,j}$ is either a local maximum or a local minimum of
$\calD^q$. We shall prove \eqref{derivative-upper-bound} for $E \in
[\wti{E}_{q,j}, E^0_{q,j}]$ with $E^0_{q,j}$ a local maximum. All
the other cases are similar.

Let $f(E)=\frac{d}{dE}\log \calD^q(E)$. It is straightforward to see
that $f'(E)<\frac{-1}{(E-\wti{E}_{q,j})^2}$ and so \beq
\label{lower-f}
f(E)=-\int_E^{E^0_{q,j}}f'(x)dx>\frac{1}{E-\wti{E}_{q,j}}-\frac{1}{E^0_{q,j}-\wti{E}_{q,j}}\eeq
(note $f(E^0_{q,j})=0$). Fix some $E' \in (\wti{E}_{q,j},E)$. It
follows that \beq \no \log
\frac{\calD^q(E)}{\calD^q(E')}=\log\calD^q(E)-\log
\calD^q(E')=\int_{E'}^{E}f(x)dx > \log
\frac{E-\wti{E}_{q,j}}{E'-\wti{E}_{q,j}}-1. \eeq Thus \beq
\frac{\calD^q(E)}{\calD^q(E')}>\frac{1}{e}\frac{E-\wti{E}_{q,j}}{E'-\wti{E}_{q,j}},
\no \eeq which implies \beq \no
e\left|\frac{\calD^q(E)}{E-\wti{E}_{q,j}}\right|>\left|\frac{\calD^q(E')-\calD^q(\wti{E}_{q,j})}{E'-\wti{E}_{q,j}}
\right|. \eeq The estimate \eqref{derivative-upper-bound} follows
from this by taking the limit $E'\rightarrow \wti{E}_{q,j}$.

We now turn to prove the lower bound. First, to fix notation, define
the intervals \beq \label{Bqj}
B_{q,j}\equiv[\wti{E}^\ell_{q,j},\wti{E}^r_{q,j}]= \left
\{\begin{array}{cc}&[E_{q,j}(0), E_{q,j}(\pi)] \quad \textrm{if }
E_{q,j}(0)<E_{q,j}(\pi) \\ &[E_{q,j}(\pi),E_{q,j}(0)] \quad
\textrm{if } E_{q,j}(\pi)<E_{q,j}(0) \end{array} \right.\eeq so that
$b_{q,j}=|B_{q,j}|$. We shall refer to the $B_{q,j}$ as the `bands'.
Let further $B^\ell_{q,j}=[\wti{E}^\ell_{q,j},\wti{E}_{q,j}]$ and
$B^r_{q,j}=[\wti{E}_{q,j},\wti{E}^r_{q,j}]$ be the two parts of the
band $B_{q,j}$ and let $b^i_{q,j}=|B^i_{q,j}|$ ($i=\ell,r$).

Note that, as $(\calD^q)'$ is a $(q-1)$ degree polynomial with simple
zeros, $|(\calD^q)'|$ has a single maximum in each interval
$[E^0_{q,j}, E^0_{q,j+1}]$ ($1 \leq j \leq q-2$). Since, for $1 \leq
j \leq (q-2)$, $B_{q,j+1} \subseteq [E^0_{q,j}, E^0_{q,j+1}]$, it
follows that $|(\calD^q)'|$ has a single maximum in each interval
$B_{q,j}$ ($1 \leq j \leq q$). If this maximum is in $B^\ell_{q,j}$
then, by monotonicity $|(\calD^q)'(\wti{E}_{q,j})| \geq
|(\calD^q)'(E)|$ for all $E \in B^r_{q,j}$ from which it follows
that \beq \no \frac{2}{b^r_{q,j}}=
\frac{|\calD^q(\wti{E}_{q,j})-\calD^q(\wti{E}^r_{q,j})|}{|\wti{E}_{q,j}-\wti{E}^r_{q,j}|}
\leq |(\calD^q)'(\wti{E}_{q,j})|.\eeq Otherwise, the maximum is in
$B^r_{q,j}$ and we get \beq \no \frac{2}{b^\ell_{q,j}} \leq
|(\calD^q)'(\wti{E}_{q,j})|.\eeq

Assume that $2\leq j \leq (q-1)$, and that $\calD^q$ is increasing
on $B_{q,j}$ and consider the polynomial $g(E)=\calD^q(E)+2$ which
has a zero at $\wti{E}^\ell_{q,j}$. An analysis similar to the one
leading to \eqref{lower-f} shows that for $E \in
(\wti{E}^\ell_{q,j},E^0_{q,j})$ \beq \no
\frac{g'(E)}{g(E)}=\frac{d}{dE}\log
g(E)>\frac{1}{E-\wti{E}^\ell_{q,j}}-\frac{1}{E^0_{q,j}-\wti{E}^\ell_{q,j}}\eeq which
implies (putting $E=\wti{E}_{q,j}$) \beq \no
\frac{(\calD^q)'(\wti{E}_{q,j})}{2}=\frac{g'(\wti{E}_{q,j})}{g(\wti{E}_{q,j})}>\frac{1}{b^\ell_{q,j}}-\frac{1}{b^\ell_{q,j}+b^r_{q,j}}=\frac{b^r_{q,j}}{b^\ell_{q,j}(b^\ell_{q,j}+b^r_{q,j})}.
\eeq Letting $t_{q,j}=\frac{2}{|(\calD^q)'(\wti{E}_{q,j})|}$, it
follows that \beq \label{quadbl}
b^\ell_{q,j}>\frac{t_{q,j}b^r_{q,j}}{b^\ell_{q,j}+b^r_{q,j}}.\eeq

By considering the function $\calD^q(E)-2$ and performing a similar
analysis, we can get the same inequality with $\ell$ and $r$
interchanged: \beq \label{quadbr}
b^r_{q,j}>\frac{t_{q,j}b^\ell_{q,j}}{b^\ell_{q,j}+b^r_{q,j}}.\eeq

We showed above that either $b^r_{q,j} \geq t_{q,j}$ or
$b^\ell_{q,j} \geq t_{q,j}$. Assume $b^r_{q,j} \geq t_{q,j}$. Then
\eqref{quadbl} implies that
$(b^\ell_{q,j})^2+b^\ell_{q,j}t_{q,j}-(t_{q,j})^2 \geq 0$, from which
it follows that $b^\ell_{q,j} \geq \frac{\sqrt{5}-1}{2}t_{q,j}$.
Similarly, $b^\ell_{q,j} \geq t_{q,j}$ implies that $b^r_{q,j} \geq
\frac{\sqrt{5}-1}{2}t_{q,j}$ so that in any case we have \beq \no
b_{q,j}=b^\ell_{q,j}+b^r_{q,j} \geq \frac{\sqrt{5}+1}{2}t_{q,j}.
\eeq Thus \beq \label{lower-b} |(\calD^q)'(\wti{E}_{q,j})| \geq
\frac{\sqrt{5}+1}{b_{q,j}}, \eeq for any $2 \leq j \leq (q-1)$. For
the case of $j=1,q$ only one of the inequalities \eqref{quadbl} or
\eqref{quadbr} can be obtained, but by monotonicity it follows that
$b^r_{q,j} \geq t_{q,j}$ corresponds to the case where
\eqref{quadbl} holds and vice versa and so we get \eqref{lower-b}
for all $1 \leq j \leq q$, which concludes the proof.
\end{proof}

\begin{proof}[Proof of Theorem \ref{Thouless-width}]
Note first that by the first equality in \eqref{Dec-Discr}, $\inf_{E
\in \bbR} |\calD^q(E+i\varepsilon)|^2\geq \inf_{E \in
[\wti{E}_{q,1},\wti{E}_{q,q}]} |\calD^q(E+i\varepsilon)|^2$. By
Lemma \ref{derivative-bound} we see that for any $E \in
[\wti{E}_{q,1},\wti{E}_{q,q}]$ \beq \label{Thouless-bound}
\left|\frac{\calD^q(E)}{E-\wti{E}_{q,j(E)}}\right|^2 \geq \left(
\frac{\sqrt{5}+1}{e}\right)^2 b_{q, j(E)}^{-2}. \eeq

Combining \eqref{central-inequality}, \eqref{Dec-Discr} and
\eqref{Thouless-bound} (with $\varepsilon=1/T$), and noticing that
$\frac{\prod_{j \neq j(E)}|E+i\varepsilon-\wti{E}_{q,j}|^2}{\prod_{j
\neq j(E)}|E-\wti{E}_{q,j}|^2} \geq 1$, we get \beq \no
\begin{split} P(q,T) &\leq 4 T^6 \left(1+2 \parallel V
\parallel_\infty \right)^2 \sup_{E \in [\wti{E}_{q,1},\wti{E}_{q,q}]}
\left(\left|\frac{\calD^q(E)}{E-\wti{E}_{q,j(E)}}\right|^{-2}\right) \\
& \leq \frac{4e^2}{(\sqrt{5}+1)^2}T^6 \left(1+2 \parallel V
\parallel_\infty \right)^2 \sup_{E \in [\wti{E}_{q,1},\wti{E}_{q,q}]}b_{q, j(E)}^2 \\
& = \frac{4e^2}{(\sqrt{5}+1)^2}T^6 \left(1+2 \parallel V
\parallel_\infty \right)^2 \left(\sup_{j}b_{q, j} \right)^2.
\end{split}
\eeq
\end{proof}

\section{Proof of Theorems \ref{clustering} and \ref{th-1.8}}

We present in this section more refined lower bounds for the
polynomials $\calD$ evaluated at a distance $1/T$ from the real
line. In particular, we examine the consequences of clustering of
the zeros of these polynomials. 

The bounds developed here are for fairly general polynomials, and we believe they may be interesting 
in other contexts as well. For this reason we depart from $\calD$ for most of the analysis and present 
our results for general polynomials (under the assumptions described below). We return to $\calD$ for the 
proofs of Theorems \ref{clustering} and \ref{th-1.8}. 

To fix notation, let $Q$ be a monic polynomial of degree $q$ with
real and simple zeros. Denote the set of zeros of $Q$ by
$Z(Q)=\{z_1, \ldots , z_q\}$. As in the previous section (see the
discussion following \eqref{Disc-form}), any point $E \in \bbR$ lies
between two extremal points of $Q$ or between an extremal point and
$\pm \infty$. Any such interval contains a unique zero of $Q$. In
this way we associate with each point $E \in \bbR$ a unique zero
$z(E)$ of $Q$. Our first lemma illustrates why
clustering of zeros implies lower bounds away from $\bbR$.

\begin{definition} \label{covered-zeros}
Let $Q$ be a polynomial of degree $q$ and let $\varepsilon>0$. We
shall say that $Q$ is {\it $\varepsilon$-covered} if $Q$ is monic
with real and simple zeros, and $Z(Q)$ is covered by a finite
collection $U_\varepsilon(Q)=\{I_j\}_{j=1}^k$ ($k \leq q$) of
disjoint closed intervals of size not exceeding $\varepsilon$ such
that every $I_j$ contains, in addition to at least one point of
$Z(Q)$, a point $x$ for which $|Q(x)|=2$. When we want to be
explicit about the family of covering intervals, we shall say that
$Q$ is $\varepsilon$-covered by $U_\varepsilon(Q)=\{I_j\}_{j=1}^k$.
\end{definition}

\begin{remark*}
In the analysis below, $|Q(x)|=2$ is not essential. With obvious modifications, it can be carried out just as well 
with any other constant. Since the relevant constant for the applications is $2$, we use it here.
\end{remark*}

Let $Q$ be an $\varepsilon$-covered polynomial and let $\{x_j\}$ be the points where 
$|Q(x)|=2$. Let $\ti{b}_{Q,j}=|x_j-z(x_j)|$ and $\ti{b}_Q=\sup_j\{ \ti{b}_{Q,j}\}$.
For $E \in \bbR$, let $I_E$ be an element of $U_\varepsilon(Q)$ containing $z(E)$. 
Finally, let $d(E)=|E-z(E)|$. We have the following lemma:

\begin{lemma}\label{lemma-3.1}
Let $\varepsilon>0$ and assume $Q$ is an $\varepsilon$-covered
polynomial. Let \beq \no A^\varepsilon=\{E \in \bbR \mid d(E) \leq
8\varepsilon \} \eeq and let \beq \no B^\varepsilon= \bbR \setminus
A^\varepsilon. \eeq

Then for $E \in A^\varepsilon$ we have \beq \label{3.1}
|Q(E+i\varepsilon)|^2 \geq \left|\frac{\varepsilon Q'(z(E))}{e}
\right|^2\left(1+\frac{1}{81} \right)^{\#\left(Z(Q)\cap I_E\right)},
\eeq and for $E \in B^\varepsilon$ \beq \label{3.2}
|Q(E+i\varepsilon)|^2 \geq \left( \frac{\varepsilon}{ \ti{b}_Q}\right)^2 9^{\# \left(Z(Q)\cap I_E
\right)}\left(\frac{1}{4}\right)^{q}.\eeq
\end{lemma}

\begin{proof}
We begin with \eqref{3.1}. Let $E \in A^\varepsilon$ and write, as
in \eqref{Dec-Discr}, \beq \no |Q(E+i\varepsilon)|^2 \geq
\varepsilon^2 \left|\frac{Q(E)}{E-z(E)}\right|^2\frac{\prod_{z_j
\neq z(E)}(E-z_j)^2+\varepsilon^2}{\prod_{z_j \neq z(E)}(E-z_j)^2}.
\eeq

The argument for the upper bound in \eqref{derivative-up-low-bound}
translates into this more general setting and we get \beq \no
\begin{split} |Q(E+i\varepsilon)|^2 & \geq \left| \frac{\varepsilon
Q'(z(E))}{e} \right|^2 \prod_{z_j \neq
z(E)}\frac{(E-z_j)^2+\varepsilon^2}{(E-z_j)^2} \\
&\geq \left| \frac{\varepsilon Q'(z(E))}{e} \right|^2 \prod_{z_j
\neq z(E),\ z_j \in I_E}\frac{(E-z_j)^2+\varepsilon^2}{(E-z_j)^2} \\
& \geq \left| \frac{\varepsilon Q'(z(E))}{e} \right|^2
\left(1+\frac{\varepsilon^2}{\left(\varepsilon+8\varepsilon
\right)^2} \right)^{\#Z(Q) \cap I_E}.
\end{split} \eeq
Equation \eqref{3.1} follows.

To prove \eqref{3.2}, fix $E \in B^\varepsilon$ and assume $E>z(E)$
(with obvious changes, everything works similarly for $E<z(E)$).
Now, since $Z(Q) \subseteq \bbR$, $|Q(E+i\varepsilon)|\geq |Q(E)|$.
Next, it follows from the definition of $z(E)$ that for any
$z(E)<y<E$, $|Q(y)|\leq |Q(E)|$. Let $\hat{E} \in I_E$ be a point
for which $|Q(\hat{E})|=2$. Clearly
$z(E)<\hat{E}+4\varepsilon<E$. Therefore, \beq \no \left|
Q(E+i\varepsilon) \right| \geq |Q(E)| \geq
|Q(\hat{E}+4\varepsilon)|. \eeq

Now \beq \no \begin{split} |Q(\hat{E}+4\varepsilon)| & \geq
\frac{|Q(\hat{E}+4\varepsilon)|}{|Q(\hat{E})|} = \\
&\left| \frac{\hat{E}+4\varepsilon-z(\hat{E})}{\hat{E}-z(\hat{E})}\right|
\prod_{z_j \in I_E,\ z_j\neq z(\hat{E})}
\left|\frac{\hat{E}+4\varepsilon -z_j}{\hat{E}-z_j} \right|
\prod_{z_j \notin I_E}
\left|\frac{\hat{E}+4\varepsilon-z_j}{\hat{E}-z_j} \right| \\
& \geq \frac{3 \varepsilon}{\ti{b}_{Q,j}}3^{\#\left(Z(Q)\cap I_E \right)-1}\prod_{z_j \notin I_E}
\left|\frac{\hat{E}+4\varepsilon-z_j}{\hat{E}-z_j} \right|,
\end{split} \eeq
since for any $y \in I_E$ $|y-\hat{E}| \leq \varepsilon$. As for the
rest of the elements of $Z(Q)$, it is clear that those of them that
are located to the left of $\hat{E}$ contribute a factor that is
greater than 1 to the right hand side. So we shall assume they are
all located to the right of $\hat{E}$. Since those that are in $I_E$
were already taken into account, it follows that we are considering
only roots that are to the right of $E$. Thus we may assume that
$|z_j-\hat{E}|>8\varepsilon$. Plugging this into the right hand side
of the last inequality and remembering that $\#Z(Q)=q$, \eqref{3.2}
follows.
\end{proof}

The estimate \eqref{3.1} implies that 'crowding together' of many
zeros of $Q$ in $I_E$'s leads to an exponential lower bound on
$|Q(E+i\varepsilon)|$ for appropriate $E$'s. If not all the zeros
are covered by a single interval of size $\varepsilon$, the estimate
\eqref{3.2} might not be useful. The following modification is
tailored to deal with this problem.

\begin{lemma}\label{lemma-3.2}
Let $1/5>\varepsilon>0$. Assume $Q$ is an $\varepsilon$-covered
polynomial and let $0<\varphi<1$ be such that
$\varepsilon^{\varphi-1}>5$. Let \beq \no A^\varphi=\{E \in \bbR
\mid d(E) \leq \varepsilon^\varphi \} \eeq and let \beq \no
B^\varphi= \bbR \setminus A^\varphi. \eeq

Then for $E \in A^\varphi$ we have \beq \label{3.3}
|Q(E+i\varepsilon)|^2 \geq \left|\frac{\varepsilon Q'(z(E))}{e}
\right|^2\left(1+\frac{\varepsilon^{2-2\varphi}}{4}
\right)^{\#\left(Z(Q)\cap I_E\right)}, \eeq and for $E \in
B^\varphi$ \beq \label{3.4} |Q(E+i\varepsilon)|^2 \geq \left(\frac{\varepsilon}{\ti{b}_Q} \right)^2 9^{\#
\left(Z(Q)\cap I_E \right)}(1-4\varepsilon^{1-\varphi})^{2q}.\eeq
\end{lemma}

\begin{proof}
To prove \eqref{3.3} we simply repeat the proof of \eqref{3.1} to
obtain \beq \no
\begin{split} |Q(E+i\varepsilon)|^2 & \geq \left| \frac{\varepsilon
Q'(z(E))}{e} \right|^2 \prod_{z_j \neq
z(E)}\frac{(E-z_j)^2+\varepsilon^2}{(E-z_j)^2} \\
&\geq \left| \frac{\varepsilon Q'(z(E))}{e} \right|^2 \prod_{z_j
\neq z(E),\ z_j \in I_E}\frac{(E-z_j)^2+\varepsilon^2}{(E-z_j)^2} \\
& \geq \left| \frac{\varepsilon Q'(z(E))}{e} \right|^2
\left(1+\frac{\varepsilon^2}{\left(\varepsilon+\varepsilon^\varphi
\right)^2} \right)^{\#Z(Q) \cap I_E},
\end{split} \eeq
which implies \eqref{3.3} since $\varphi<1$.

To prove \eqref{3.4}, fix $E \in B^\varphi$ and assume $E>z(E)$
(again, obvious changes should be made for $E<z(E)$). As before,
let $\hat{E} \in I_E$ be a point for which $|Q(\hat{E})|=2$.
Since $\varepsilon^{\varphi}>5\varepsilon$,
$z(E)<\hat{E}+4\varepsilon<E$. Thus \beq \no \left|
Q(E+i\varepsilon) \right| \geq |Q(E)| \geq
|Q(\hat{E}+4\varepsilon)|, \eeq and as before \beq \no
\begin{split} |Q(\hat{E}+4\varepsilon)| \geq \frac{\varepsilon}{\ti{b}_{Q,j}} 3^{\#\left(Z(Q)\cap I_E \right)}\prod_{z_j \notin I_E}
\left|\frac{\hat{E}+4\varepsilon-z_j}{\hat{E}-z_j} \right|.
\end{split} \eeq
As in the proof of \eqref{3.2}, assuming that $z_j \notin I_E$
satisfy $|z_j-\hat{E}|>|E-\hat{E}|\geq\varepsilon^\varphi$ and
remembering that $\#Z(Q)=q$, we get \eqref{3.4}.
\end{proof}

\begin{definition} \label{alphaxi-clustered}
Let $0<\varepsilon<1$ and $0<\xi \leq 1$. We say that the polynomial
$Q$ is {\it $(\varepsilon,\xi)$-clustered} if $Q$ is
$\varepsilon$-covered by some $U_\varepsilon(Q)=\{I_j\}_{j=1}^k$, so
that for any $I_j \in U_\varepsilon(Q)$, $\# \left(Z(Q)\cap I_j
\right) \geq q^\xi$. When we want to be explicit about the covering
set, we shall say that $Q$ is {\it $(\varepsilon,\xi)$-clustered} by
$U=\{I_j\}_{j=1}^k$.
\end{definition}

\begin{lemma} \label{lemma-3.3}
Let $0<\alpha<1$ and $2/3< \xi \leq 1$. Then there exist
$\delta=\delta(\alpha,\xi)>0$, $q_0=q_0(\alpha,\xi,\delta)>0$ and a
universal constant $C>0$ such that any
$(q^{-1/\alpha},\xi)$-clustered polynomial $Q$ with $q \equiv
\textrm{deg}(Q) \geq q_0$ satisfies

\beq \label{3.5} \inf_{E \in \bbR}|Q(E+i q^{-1/\alpha})|^2 \geq
\frac{1}{e^2} \left( \min \left \{\min_{z \in Z(Q)} |Q'(z)|, \left(\ti{b}_Q \right)^{-1} \right\} \right)^2 e^{C
q^\delta}\eeq

\end{lemma}

\begin{proof}
If $\xi=1$ then by assumption there exists a single interval of size
$q^{-1/\alpha}$ containing all the zeros of $Q$. It then follows
from Lemma \ref{lemma-3.1} that \beq \no 
\begin{split} & |Q(E+iq^{-1/\alpha})|^2 \\
& \geq \min \left( \left( \frac{q^{-1/\alpha} \min_{z \in Z(Q)}
|Q'(z)|}{e} \right)^2\left(\frac{82}{81} \right)^q,\ \left(\frac{q^{-1/\alpha}}{\ti{b}_Q} \right)^2
\left( 
\frac{9}{4}\right)^{q} \right). \end{split}\eeq It follows that \eqref{3.5}
holds with any $\delta<1$ for sufficiently large $q$.

Assume now that $1>\xi>2/3$ and let $\varepsilon \equiv
q^{-1/\alpha}$. Pick $\delta>0$ so that $\delta<\min \left(
\frac{3}{2}\xi-1, 1-\xi \right)$. It follows that \beq \no
0<\frac{3}{2}\xi-1<\frac{\xi}{2}-\delta<1. \eeq

Let $\varphi=1-\frac{\alpha \xi}{2}+\alpha \delta$. Then \beq \no
0<\varphi<1 \eeq and, for $q$ sufficiently large, \beq \no
\varepsilon^{\varphi-1}=q^{1/\alpha-\varphi/\alpha}=
q^{\frac{\xi}{2}-\delta}>q^{\frac{3\xi}{2}-1}>5. \eeq

By Lemma \ref{lemma-3.2} we have, for  $E \in A^\varphi$, \beq \no
\begin{split}
|Q(E+i\varepsilon)|^2 \geq \left(\frac{q^{-1/\alpha} \min_{z \in
Z(Q)} |Q'(z)|}{e}
\right)^2\left(1+\frac{1}{4q^{(2-2\varphi)/\alpha}} \right)^{q^\xi}.
\end{split}
\eeq Since $\left(1+\frac{1}{4q^{(2-2\varphi)/\alpha}}\right)<2$ we
may use $\log(1+x)\geq \frac{x}{1+x}$ to obtain \beq \no
\begin{split}
|Q(E+i\varepsilon)|^2 &\geq \left(\frac{q^{-1/\alpha} \min_{z \in
Z(Q)} |Q'(z)|}{e}
\right)^2\left(e^{1/8}\right)^{q^{\left( \xi+2\varphi/\alpha-2/\alpha \right)}} \\
&=\left(\frac{q^{-1/\alpha} \min_{z \in Z(Q)} |Q'(z)|}{e} \right)^2
e^{\frac{q^{2\delta}}{8}}.
\end{split}
\eeq

In case $E \in B^\varphi$ we have from Lemma \ref{lemma-3.2}

\beq \no |Q(E+i\varepsilon)|^2 \geq \left( \frac{q^{-1/\alpha}}{\ti{b}_Q} \right)^2
9^{q^\xi}(1-4q^{\varphi/\alpha-1/\alpha})^{2q}=
\left( \frac{q^{-1/\alpha}}{\ti{b}_Q} \right)^2
9^{q^\xi}(1-4q^{\delta-\xi/2})^{2q}.\eeq Again, by $\log(1+x)\geq
\frac{x}{1+x}$ and for $q$ sufficiently large \beq \no \begin{split}
|Q(E+i\varepsilon)|^2 & \geq
\left( \frac{q^{-1/\alpha}}{\ti{b}_Q} \right)^2
9^{q^\xi}e^{\frac{-8q^{1+\delta-\xi/2}}{(1-4q^{\delta-\xi/2})}} \\
& \geq \left( \frac{q^{-1/\alpha}}{\ti{b}_Q} \right)^2
e^{q^\xi}e^{-16q^{1+\delta-\xi/2}} \\
& \geq \left( \frac{q^{-1/\alpha}}{\ti{b}_Q} \right)^2
 e^{q^\xi
\left(1-16q^{1+\delta-3\xi/2} \right) } \geq \left( \frac{q^{-1/\alpha}}{\ti{b}_Q} \right)^2
 e^{\frac{q^\xi}{2}},
\end{split} \eeq
since $1+\delta-3\xi/2<0$.

Thus, since $\delta<\xi$, we see that for $q$ sufficiently large
\beq \no |Q(E+iq^{-1/\alpha})|^2 \geq \frac{1}{e^2} \left( \min
\left\{\min_{z \in Z(Q)} |Q'(z)|, \left( \ti{b}_Q \right)^{-1} \right \} \right)^2 e^{\frac{q^\delta}{8}}
\eeq and we are done.
\end{proof}

\begin{proof}[Proof of Theorem \ref{clustering}]
By Lemma \ref{first-lemma}, \beq \no P(q,T) \leq 4 T^4 \left(1+2
\parallel V
\parallel_\infty \right)^2 \inf_{E \in \bbR} \left |\calD^q(E+i/T)\right|^{-2}. \eeq
Since $T^{-1} \geq q^{-1/\alpha}$, $\left| \calD^q(E+i/T) \right|
\geq \left| \calD^q(E+iq^{-1/\alpha}) \right|$. By assumption, there
exists a collection $U(\calD^q)=\{I_1,\ldots,I_k\}$ of disjoint intervals,
each of size not exceeding $q^{-1/\alpha}$ such that each one of
these intervals contains at least $q^\xi$ zeros of $\calD^q$. Since
$q^\xi\geq 2$ for $q$ large enough, we see from property 4 of
$\calD^q$ (quoted at the beginning of Section 2) that $\calD^q$ is
$(q^{-1/\alpha},\xi)$-clustered. Therefore, by Lemma
\ref{lemma-3.3}, we see that \beq \label{3.} \inf_{E \in
\bbR}|\calD^q(E+i q^{-1/\alpha})|^2 \geq \frac{1}{e^2} \left(
\min\left\{\min_{j} |(\calD^q)'(\wti{E}_{q,j})|, \left(\ti{b}_{\calD^q} \right)^{-1} \right \} \right)^2 e^{C q^\delta},\eeq
where $C$ is universal and $\delta$ depends on $\alpha$ and $\xi$.
Now, clearly $\ti{b}_{\calD^q,j}\leq b_{q,j}$ so obviously $\ti{b}_Q \leq \sup_{1 \leq j \leq q}b_{q,j}$.
By Lemma \ref{derivative-bound}, this implies that \beq \no
\begin{split} P(q,T) \leq
4e^2 T^4 \left(1+2\parallel V\parallel_\infty
\right)^2 \sup_{1\leq j \leq q}\left(b_{q,j}\right)^2 e^{-C q^\delta}.
\end{split} \eeq
This is \eqref{cluster-bound}.
\end{proof}

A serious limitation of Theorem \ref{clustering} is the assumed
lower bound on the clustering strength, that is, the requirement
that $\xi>2/3$. The reason for this requirement is the fact that one
needs to overcome an exponentially decreasing factor in the estimate
of $|Q(E+i\varepsilon)|$ for $E \in B^\varphi$ (see \eqref{3.2} and
\eqref{3.4}). The issue here is the absence of a lower bound on the
distance of the zeros of $Q$ that are to the right of $E$. Such a
lower bound, however, could be obtained by assuming an upper bound
on the clustering strength and a lower bound on the cluster sizes.
The following lemma shows that the existence of such bounds
simultaneously on different scales allows us to consider any
$\xi>0$. We first recall Definition \ref{def-1.4} in the context of
general polynomials.

\begin{definition} \label{def-3.6}
We say that a sequence of polynomials, $\{Q_\ell\}_{\ell=1}^\infty$,
with $\textrm{deg}(Q_\ell) \equiv q_\ell \underset{l \rightarrow
\infty}{\rightarrow}\infty$, is \emph{uniformly clustered} if
$Q_\ell$ is $\{q_\ell^{-1/\alpha_\ell}, \xi_\ell \}$-clustered by
$U_\ell=\{I_j^\ell \}_{j=1}^{k_\ell}$ and the following hold:
\\
(i) If $\ell_1<\ell_2$ then $q_{\ell_1}^{-1/\alpha_{\ell_1}}>q_{\ell_2}^{-1/\alpha_{\ell_2}}$. \\
(ii) There exists $\mu \geq 1$ and a constant $C_1>0$ so that \beq
\no \inf_{1 \leq j \leq k_\ell} |I^\ell_j| \geq C_1
q_\ell^{-\mu/\alpha_\ell}.
\eeq \\
(iii) There exists $\delta>0$ so that $ \delta < \xi_\ell
<(1-\delta)$,
and $\delta< \alpha_\ell$ for all $\ell$. \\
(iv) If we define ${\bar{\xi}_\ell}$ by $\sup_{1 \leq j \leq k_\ell}
\# \left(Z(Q_\ell)\cap I_j^\ell \right) \leq q_\ell^{\bar{\xi}_\ell}$,
then there exists a constant $C_2$ so that \beq \no \bar{\xi}_\ell-
\xi_\ell \leq
\frac{C_2}{\log q_\ell}.\eeq \\
As above, when we want to be explicit about the cover and the
relevant exponents we shall say that $\{Q_\ell\}_{\ell=1}^\infty$ is
uniformly clustered by $U_\ell=\{I_j^\ell\}_{j=1}^{k_\ell}$ with
exponents $\{\alpha_\ell, \xi_\ell, \mu \}$.
\end{definition}

\begin{lemma} \label{lemma-3.8}
Let $\{Q_\ell\}_{\ell=1}^\infty$ be a sequence of polynomials with
$\textrm{deg}(Q_\ell) \equiv q_\ell \underset{\ell \rightarrow
\infty}{\rightarrow}\infty,$ that is uniformly clustered by
$U_\ell=\{I_j^\ell\}_{j=1}^{k_\ell}$ with exponents $\{\alpha_\ell,
\xi_\ell, \mu \}$. Suppose, moreover, that
$\{U_\ell\}_{\ell=1}^\infty$ scales nicely with exponents $\mu$ and
$\omega$ for some $0<\omega<1$. Assume also that, for some $\zeta>0$
\beq \label{eq-3.7} 2\omega \left(\frac{\mu-1}{\mu-\omega}
\right)+\zeta<\xi_\ell \alpha_\ell. \eeq Finally, suppose that $\liminf_{\ell
\rightarrow \infty} \left( \inf_{z \in Z(Q_\ell)} |Q_\ell'(z)|
\right)>0$ and that $\ti{b}_{Q_\ell}$ is bounded from above.

Then for any $m >0$ \beq \label{eq-3.8} \lim_{\ell
\rightarrow \infty} \left(\inf_{E \in \bbR} |Q_\ell(E+i
q_\ell^{-1/\alpha_\ell})| q_\ell^{-m/\alpha_\ell}
\right)=\infty. \eeq
\end{lemma}

\begin{proof}
Let $\varepsilon_\ell=q_\ell^{-1/\alpha_\ell}$. Choose $\delta>0$ so
that $\liminf_\ell \frac{\xi_\ell}{2}-\delta>0$ and $\frac{\zeta}{2 \alpha_\ell}>\delta$ 
for all $\ell$, and let
$\varphi_\ell=1-\frac{\alpha_\ell \xi_\ell}{2}+\alpha_\ell \delta$.
It follows that $0<\varphi_\ell<1$.

As in Lemma \ref{lemma-3.2}, let \beq \no A^{\varphi_\ell}=\{E \in
\bbR \mid d(E) \leq \varepsilon_\ell^{\varphi_\ell} \} \eeq and let
\beq \no B^{\varphi_\ell}= \bbR \setminus A^{\varphi_\ell}. \eeq
Then for large enough $\ell$ and for $E \in A^{\varphi_\ell}$, as in
the proof of Lemma \ref{lemma-3.2}, \beq \no
|Q_\ell(E+i\varepsilon_\ell)|^2 \geq
\left(\frac{q_\ell^{-1/\alpha_\ell} \min_{z \in Z(Q_\ell)}
|Q_\ell'(z)|}{e} \right)^2 e^{\frac{q_\ell^{2\delta}}{8}}. \eeq
Thus, since $\alpha_\ell$ and $\min_{z \in Z(Q_\ell)} |Q_\ell'(z)|$
are bounded away from zero, we only need to obtain a lower bound for
$E \in B^{\varphi_\ell}$ for sufficiently large $\ell$.

So let $E \in B^{\varphi_\ell}$. As before, since $\varphi_\ell$ is
bounded away from $1$, if $\ell$ is large enough
$\varepsilon_\ell^{\varphi_\ell}>5 \varepsilon_\ell$. Let
$\{z_j^\ell\}_{j=1}^{q_\ell}$ denote the zeros of $Q_\ell$ and let
$z^\ell(E)$ be the unique zero associated with $E$ according to the
discussion at the beginning of this section. Assume $z^\ell(E)<E$
and let $\hat{E}\in I_E^\ell$ be a point with
$|Q_\ell(\hat{E})| = 2$. As before \beq \label{eq-3.9}
\begin{split} |Q_\ell(E+i\varepsilon_\ell)| & \geq |Q(\hat{E}+4\varepsilon_\ell)| \geq 
\frac{q_\ell^{-1/\alpha_\ell}}{\ti{b}_{Q_\ell}} 3^{\#\left(Z(Q_\ell)\cap I_E^\ell \right)}\prod_{z^\ell_j \notin I^\ell_E}
\left|\frac{\hat{E}+4\varepsilon_\ell-z^\ell_j}{\hat{E}-z^\ell_j}
\right| \\
& \geq \frac{q_\ell^{-1/\alpha_\ell}}{\ti{b}_{Q_\ell}} e^{q_\ell^{\xi_\ell}}\prod_{z^\ell_j \notin I^\ell_E}
\left|\frac{\hat{E}+4\varepsilon_\ell-z^\ell_j}{\hat{E}-z^\ell_j}
\right|.
\end{split} \eeq
We shall use the fact that $\{Q_\ell\}$ is uniformly clustered with
nicely scaling covering sets to obtain lower bounds on
$\prod_{z^\ell_j \notin I^\ell_E}
\left|\frac{\hat{E}+4\varepsilon_\ell-z^\ell_j}{\hat{E}-z^\ell_j}
\right|$. As in the proof of Lemma \ref{lemma-3.2}, we shall assume
that the $z_j^\ell$ that are not in $I^\ell_E$ are located to the
right of $E$.

Let $M$ be an integer so large that $\frac{\varphi_\ell}{\alpha_\ell
M}<\xi_\ell+\frac{1}{\alpha_\ell}-1-\eta$ for some $\eta>0$ for all
$\ell$. Such an $M$ exists since $\alpha_\ell<1$ and $\xi_\ell$ is
bounded away from $0$. Let
$\hat{\varepsilon}=\varepsilon_\ell^{\varphi_\ell/(M\mu)}$. Then,
since $\{U_\ell\}_{\ell=1}^\infty$ scale nicely, there is a
collection of intervals $U_{\hat{\varepsilon}}$ of length at most
$\hat{\varepsilon}$ and at least
$C_1\varepsilon_\ell^{\varphi_\ell/M}$ ($C_1$ some positive
constant) covering the elements of $U_\ell$ as in Definition
\ref{def-1.5}. In particular, it is possible to cover
$J\equiv\cup_{j=1}^{k_\ell}I_j^\ell \cap [E,
E+C_1\varepsilon_\ell^{\varphi_\ell/M}]$ by using no more than two
elements of $U_{\hat{\varepsilon}}$.

We proceed to analyze the possible distribution of zeros of $Q_\ell$
in $J$. Let $1\leq r <M$. Note that, by the nice scaling property,
there exists a constant $C_3>0$ such that any element of
$U_{\hat{\varepsilon}^r}$ contains no more than $ C_3 \left(
\hat{\varepsilon}^r/\varepsilon_\ell \right)^\omega=C_3
\varepsilon_\ell^{\omega\left((\varphi_\ell r)/(M\mu)-1\right)}$ elements of $U_{\varepsilon_\ell}$. By
the uniform clustering, each set in $U_{\varepsilon_\ell}$ holds no more than
$q_\ell^{\bar{\xi}_\ell}$ zeros of $Q_\ell$. Therefore, each
interval of $U_{\hat{\varepsilon}^r}$ holds no more than $C_3
q_\ell^{\bar{\xi}_\ell}\varepsilon_\ell^{\omega\left((\varphi_\ell
r)/(M\mu)-1\right)} =C_3
q_\ell^{\bar{\xi}_\ell-\frac{\omega}{\alpha_\ell}((\varphi_\ell
r)/(M\mu)-1)}$ zeros of $Q_\ell$.

Now, there are no more than $C_3\left(
\frac{\hat{\varepsilon}}{\hat{\varepsilon}^2} \right)^\omega=C_3
\varepsilon_\ell^{-(\varphi_\ell \omega)/(M\mu)}$ elements of
$U_{\hat{\varepsilon}^2}$ in each interval of
$U_{\hat{\varepsilon}}$. Therefore, it takes no more than $2C_3
\varepsilon_\ell^{-(\varphi_\ell \omega)/(M\mu)}$ elements of
$U_{\hat{\varepsilon}^2}$ to cover $J$. Each of these has length at
least $C_1 \hat{\varepsilon}^{2\mu}=C_1 \varepsilon_\ell^{2
\varphi_\ell/M}$ and holds no more than $C_3
q_\ell^{\bar{\xi}_\ell-\frac{\omega}{\alpha_\ell}((2\varphi_\ell)/(M\mu)-1)}$ zeros of $Q_\ell$. 
Take the two intervals closest to $E$ (so that at least one of them lies completely to the right of $E$).
Each of these intervals contains no more than $C_3
\varepsilon_\ell^{-(\varphi_\ell \omega)/(M\mu)}$ elements of
$U_{\hat{\varepsilon}^3}$, each of length at least $C_1
\varepsilon_\ell^{3 \varphi_\ell/M}$ and containing at most $C_3
q_\ell^{\bar{\xi}_\ell-\frac{\omega}{\alpha_\ell}((3\varphi_\ell
)/(M\mu)-1)}$ zeros of $Q_\ell$. Of these elements of
$U_{\hat{\varepsilon}^3}$ (contained in the two intervals above), take the two closest to $E$ 
(again, so that at least one lies completely to the right of $E$) and decompose
them as above with the elements of $U_{\hat{\varepsilon}^4}$ .
Continue in this manner up to $U_{\hat{\varepsilon}^M}$. From the
picture we have just described we get the following estimate:

\beq \label{eq-3.10} \prod_{z^\ell_j \notin I^\ell_E}
\left|\frac{\hat{E}+4\varepsilon_\ell-z^\ell_j}{\hat{E}-z^\ell_j}
\right|\geq \Theta_1(\ell) \Theta_2(\ell), \eeq where (letting
$\ti{C}_1=\min(1,C_1)$ ) \beq \no
\begin{split} \Theta_1(\ell) & =\prod_{j=1}^{\left[ 2C_3
\varepsilon_\ell^{-(\varphi_\ell \omega)/(M\mu)} \right]+1}
\left(1-\frac{4 \varepsilon_\ell}{j\ti{C}_1
\varepsilon_\ell^{\varphi_\ell}} \right)^{2C_3
q_\ell^{\bar{\xi}_\ell-\frac{\omega}{\alpha_\ell}\left(\frac{\varphi_\ell
}{\mu}-1\right )}} \\
& \cdot \prod_{j=1}^{\left[ 2C_3 \varepsilon_\ell^{-(\varphi_\ell
\omega)/(M\mu)} \right]+1} \left(1-\frac{4
\varepsilon_\ell}{j\ti{C}_1
\varepsilon_\ell^{\frac{M-1}{M}\varphi_\ell}} \right)^{2C_3
q_\ell^{\bar{\xi}_\ell-\frac{\omega}{\alpha_\ell}\left(\frac{(M-1)\varphi_\ell
}{M\mu}-1\right)}} \\
& \cdot \prod_{j=1}^{\left[ 2C_3 \varepsilon_\ell^{-(\varphi_\ell
\omega)/(M\mu)} \right]+1} \left(1-\frac{4
\varepsilon_\ell}{j\ti{C}_1
\varepsilon_\ell^{\frac{M-2}{M}\varphi_\ell}} \right)^{2C_3
q_\ell^{\bar{\xi}_\ell-\frac{\omega}{\alpha_\ell}\left(\frac{(M-2)\varphi_\ell
}{M\mu}-1\right)}} \\
& \cdots \prod_{j=1}^{\left[ 2C_3 \varepsilon_\ell^{-(\varphi_\ell
\omega)/(M\mu)} \right]+1} \left(1-\frac{4
\varepsilon_\ell}{j\ti{C}_1
\varepsilon_\ell^{\frac{2}{M}\varphi_\ell}} \right)^{2C_3
q_\ell^{\bar{\xi}_\ell-\frac{\omega}{\alpha_\ell}\left(\frac{(2)\varphi_\ell
}{M\mu}-1\right)}},
\end{split} \eeq
and \beq \no \Theta_2(\ell)= \left(1 -\frac{4\varepsilon_\ell}{C_1
\varepsilon_\ell^{\frac{\varphi_\ell}{M}}} \right)^{q_\ell}. \eeq

The term $\Theta_2(\ell)$ comes from the zeros of $Q_\ell$ outside $J$.
The term $\Theta_1(\ell)$ comes from the contribution of the zeros inside
$J$. The first term in the product defining $\Theta_1(\ell)$, i.e.\ $\prod_{j=1}^{\left[ 2C_3
\varepsilon_\ell^{-(\varphi_\ell \omega)/(M\mu)} \right]+1}
\left(1-\frac{4 \varepsilon_\ell}{j\ti{C}_1
\varepsilon_\ell^{\varphi_\ell}} \right)^{2C_3
q_\ell^{\bar{\xi}_\ell-\frac{\omega}{\alpha_\ell}\left(\frac{\varphi_\ell
}{\mu}-1\right )}}$, comes from the contribution of the zeros inside
the two intervals of $U_{\hat{\varepsilon}^{(M-1)}}$ that are closest to
$E$ from the right (one of them may contain $E$ and extend to its left as well). 
The zeros inside each of these intervals have to be
distributed among at most $\left[ C_3
\varepsilon_\ell^{-(\varphi_\ell \omega)/(M\mu)} \right]+1$ elements
of $U_{\hat{\varepsilon}^{M}}$ with at most ${C_3
q_\ell^{\bar{\xi}_\ell-\frac{\omega}{\alpha_\ell}\left(\frac{\varphi_\ell
}{\mu}-1\right )}}$ in each element. In the same way, the second
term in the product comes from estimating the contribution of the
zeros inside the two intervals of $U_{\hat{\varepsilon}^{(M-2)}}$ that are
closest to $E$ from the right. Continuing in this manner we obtain
$\Theta_1(\ell)$. The factor $2$ in the exponent $\left( 2C_3
q_\ell^{\bar{\xi}_\ell-\frac{\omega}{\alpha_\ell}\left(\frac{\varphi_\ell
}{\mu}-1\right )}\right)$ comes from the fact that intervals in the various $U_\varepsilon$ may intersect 
each other. 
Once again, we recall that we assume that all zeros outside 
$I_E^\ell$ are to the right of $E$.

By \eqref{eq-3.9} we only have to show $\liminf_{\ell \rightarrow
\infty}\frac{\log \Theta_1(\ell)}{q_\ell^{\xi_\ell}} \geq 0$,
$\liminf_{\ell \rightarrow \infty}\frac{\log
\Theta_2(\ell)}{q_\ell^{\xi_\ell}} \geq 0$. A straightforward computation
shows that the choice of $M$ insures that this indeed is the case
for $\Theta_2(\ell)$ so we are left with estimating $\liminf_{\ell
\rightarrow \infty}\frac{\log \Theta_1(\ell)}{q_\ell^{\xi_\ell}} \geq 0$.
Now, for $\ell$ large enough, \beq \label{eq-3.11} \begin{split}
\frac{\log \Theta_1(\ell)}{q_\ell^{\xi_\ell}} & \geq
-\frac{16 C_3}{q_\ell^{\xi_\ell}} \sum_{s=0}^{M-2} \sum_{j=1}^{\left[
2C_3 \varepsilon_\ell^{-(\varphi_\ell \omega)/(M\mu)} \right]+1}
\frac{\varepsilon_\ell^{\left(1-\frac{M-s}{M}\varphi_\ell
\right)}}{j\ti{C}_1}
q_\ell^{\bar{\xi}_\ell-\frac{\omega}{\alpha_\ell}\left(\frac{(M-s)\varphi_\ell
}{M\mu}-1\right)}\\
& \geq -C \sum_{s=0}^{M-2}
q_\ell^{\frac{-1}{\alpha_\ell}\left(1-\frac{M-s}{M}\varphi_\ell
\right)}
q_\ell^{\bar{\xi}_\ell-\xi_\ell-\frac{\omega}{\alpha_\ell}\left(\frac{(M-s)\varphi_\ell
}{M\mu}-1\right)} \sum_{j=1}^{\left[ 2C_3 q_\ell^{(\varphi_\ell
\omega)/(\alpha_\ell M\mu)} \right]+1}
\frac{1}{j} \\
& \geq -\ti{C} \log \left( q_\ell \right)
q_\ell^{\bar{\xi}_\ell-\xi_\ell}
q_\ell^{\frac{\varphi_\ell}{\alpha_\ell}-\frac{1}{\alpha_\ell}+\frac{\omega}{\alpha_\ell}-\frac{\omega\varphi_\ell}{\mu
\alpha_\ell}} \sum_{s=0}^{M-2} q_\ell^{\left(-\frac{\varphi_\ell}{M
\alpha_\ell}+\frac{\omega \varphi_\ell }{\alpha_\ell \mu M} \right)s} \\
& \geq -\ti{C} \log \left( q_\ell \right)
q_\ell^{\bar{\xi}_\ell-\xi_\ell}
q_\ell^{\frac{\varphi_\ell}{\alpha_\ell}-\frac{1}{\alpha_\ell}+\frac{\omega}{\alpha_\ell}-\frac{\omega\varphi_\ell}{\mu
\alpha_\ell}} \frac{1}{1-q_\ell^{\frac{\varphi_\ell}{\alpha_\ell
M}\left(\frac{\omega}{\mu}-1 \right)}} \\
&\geq -\ti{\ti{C}} q_\ell^{\bar{\xi}_\ell-\xi_\ell}
q_\ell^{\frac{1}{\alpha_\ell}\left(\varphi_\ell
\left(1-\frac{\omega}{\mu} \right)+\omega-1 \right)} \log(q_\ell)
\end{split} \eeq
where $C,\ \ti{C},\ \ti{\ti{C}}$ are some positive constants. By
(iv) in Definition \ref{def-3.6} we see that
$q_\ell^{\bar{\xi}_\ell-\xi_\ell} \leq \hat{C}$ for some positive
constant $\hat{C}$. Moreover, a straightforward computation shows
that \eqref{eq-3.7}, together with $\frac{\zeta}{2 \alpha_\ell}>\delta$, 
implies that $\limsup_{\ell \rightarrow \infty}
{\frac{1}{\alpha_\ell}\left(\varphi_\ell \left(1-\frac{\omega}{\mu}
\right)+\omega-1 \right)}<0$. This, together with \eqref{eq-3.11},
implies that $\liminf_{\ell \rightarrow \infty}\frac{\log
\Theta_1(\ell)}{q_\ell^{\xi_\ell}} \geq 0$ and we are done.
\end{proof}

\begin{proof}[Proof of Theorem \ref{th-1.8}]
By Lemma \ref{first-lemma}, \beq \no P(q,T) \leq 4 T^4 \left(1+2
\parallel V
\parallel_\infty \right)^2 \left( \frac{1}{\inf_{E \in \bbR} \left|\calD^q(E+i/T) \right|}\right)^2. \eeq
If $T^{-1} \geq q_\ell^{-1/\alpha_\ell}$, $\left| \calD^{q_\ell}(E+i/T) \right|
\geq \left| \calD^{q_\ell}(E+iq_\ell^{-1/\alpha_\ell}) \right|$. Thus, we only need to check that
$\calD^{q_\ell}$ satisfies the assumptions of Lemma \ref{lemma-3.8}. We know 
(see the discussion after Corollary \ref{cor-1.3})
$\sup_{1\leq j \leq q_\ell} b_{q_\ell,j}$ is bounded from above, which by Lemma \ref{derivative-bound}, implies that 
$\min_{1 \leq j \leq q_\ell}\left| \left( \calD^{q_\ell} \right)'\left(\widetilde{E}_{q_\ell,j} \right) \right|$ is bounded 
away from zero. The rest of the properties are clear from the assumptions of Theorem \ref{th-1.8}. 
\end{proof}

\section{Proof of Corollaries \ref{cor-1.3}, \ref{cor-1.6}, \ref{cor-1.10} and  Proposition \ref{full-line}}

The corollaries follow from 
\begin{proposition} \label{prop-4.1} 
Let $\alpha>0$. Then 
\begin{enumerate}
\item If there exist a monotone increasing sequence $\{q_\ell\}_{\ell=1}^\infty$, a constant $C>0$ and an
$\varepsilon>0$ such that $P(q_\ell,q_\ell^{1/\alpha}) \leq C q_\ell^{-\varepsilon}$, then $\alpha_l^- \leq \alpha$.

Moreover, if $\{q_\ell\}_{\ell=1}^\infty$ is exponentially growing and 
$P(q_\ell,T) \leq C q_\ell^{-\varepsilon}$ for any $T \leq q_\ell^{1/\alpha}$, then $\alpha_l^+ \leq \alpha$.

\item If there exists a monotone increasing sequence $\{q_\ell\}_{\ell=1}^\infty$ 
such that $P(q_\ell,q_\ell^{1/\alpha})=O(q_\ell^{-m})$ for all $m$, then $\alpha_u^- \leq \alpha$.

Moreover, if $\{q_\ell\}_{\ell=1}^\infty$ is exponentially growing and $P(q_\ell,T)=O(q_\ell^{-m})$ for all $m$ uniformly in 
$T \leq q_\ell^{1/\alpha}$, then $\alpha_u^+ \leq \alpha$. 
\end{enumerate}
\end{proposition}

\begin{proof}
The proof of the bounds for $\alpha_l^{\pm}$ (part 1 in the Proposition) and $\alpha_u^{\pm}$ 
(part 2 in the proposition) are almost identical so we give the details only for part 1. 
\begin{enumerate}
\item Let $T_\ell=q_\ell^{1/\alpha}$. Then 
\beq \no
\liminf_{T \rightarrow \infty} \frac{\log P(T^\alpha, T)}{\log T} \leq
\liminf_{\ell \rightarrow \infty} \frac{\log P(T_\ell^\alpha, T_\ell)}{\log T_\ell} \leq -\varepsilon \alpha < 0  
\eeq so $\alpha_l^- \leq \alpha$ by definition. 

Now assume $\{q_\ell\}$ is exponentially growing. We shall show that $\alpha'>\alpha_l^+$ for any $\alpha'>\alpha$. 
Our proof here follows the strategy implemented in \cite[Proof of Theorem 3]{DT3} for the construction of $N(T)$. 
Let $\alpha'>\alpha$. Let $\overline{q}=\sup_\ell \frac{q_{\ell+1}}{q_\ell}$ and $\underline{q}=\inf_\ell \frac{q_{\ell+1}}{q_\ell}$ and 
$r=\frac{\log \overline{q}}{\log \underline{q}}$.
Finally, for any sufficiently large $T$ let $\ell(T)$ be the unique index such that 
\beq \no
q_{\ell(T)} \leq T^\alpha < q_{\ell(T)+1}
\eeq
and let $q(T)=q_{\ell(T)+\lfloor\sqrt{\ell(T)} \rfloor}$. Note that 
\beq \no 
\frac{q(T)}{q_{\ell(T)}} \leq \overline{q}^{\sqrt{\ell(T)}} \leq \underline{q}^{r\sqrt{\ell(T)}}
\eeq
and $\underline{q}^{\ell(T)} \leq \ti{C} q_{\ell(T)} \leq \ti{C} T^\alpha$ for some constant $\ti{C}>0$, so 
\beq \no
q(T) \leq \ti{C} q_{\ell(T)} T^{\alpha r/\sqrt{\ell(T)}} \leq \ti{C} T^\alpha T^{\alpha r/\sqrt{\ell(T)}}.
\eeq 
Since $\ell(T) \rightarrow \infty$, it follows that for any $\delta>0$ there exists a constant $C_\delta$ such that 
\beq \no
q(T) \leq C_\delta T^{\alpha+\delta}. 
\eeq
Pick $\delta$ so that $\alpha+\delta < \alpha'$. Thus, for sufficiently large $T$, 
$P(T^{\alpha'},T) \leq P(C_\delta T^{\alpha+\delta},T)\leq P(q(T),T) \leq P(q_{\ell(T)+1},T)$ and we get (recall 
$q_{\ell(T)}^{1/\alpha}\leq T < q_{\ell(T)+1}^{1/\alpha}$)
\beq \no
\limsup_{T \rightarrow \infty} \frac{\log P(T^{\alpha'})}{\log T} \leq 
\limsup_{\ell \rightarrow \infty} \frac{-\varepsilon \alpha \log q_{\ell(T)+1}}{\log q_{\ell(T)}}=-\varepsilon \alpha<0.
\eeq
Therefore $\alpha'>\alpha_l^+$ and we are done.
\item Repeat the proof of part 1 with the changes: $\alpha_l^-$ to $\alpha_u^-$, $\alpha_l^+$ to $\alpha_u^+$, and replace 
$-\varepsilon \alpha$ by $-\infty$.   
\end{enumerate}
\end{proof}

\begin{proof}[Proof of Corollary \ref{cor-1.3}]
Let $\alpha>3/\beta$. Clearly, by Theorem \ref{Thouless-width}, for any $T \leq q_\ell^{1/\alpha}$,
$P(q_\ell,T) \leq C(V)q_\ell^{\frac{6}{\alpha}-2\beta}$, where $C(V)$ is a constant that depends on the potential 
$V$. By the assumption, $\frac{6}{\alpha}-2\beta<0$ so we may apply part 1 of Proposition \ref{prop-4.1} to conclude that 
$\alpha_l^- \leq \alpha$ and so that $\alpha_l^- \leq 3/\beta$. If $\{q_\ell\}_{\ell=1}^\infty$ is exponentially growing 
we get $\alpha_l^+ \leq \alpha$ which implies that $\alpha_l^+ \leq 3/\beta$. 
\end{proof}

\begin{proof}[Proof of Corollary \ref{cor-1.6}]
The corollary follows immediately from Theorem \ref{clustering} and Proposition \ref{prop-4.1} above.
\end{proof}

\begin{proof}[Proof of Corollary \ref{cor-1.10}]
The corollary follows immediately from Theorem \ref{th-1.8} and Proposition \ref{prop-4.1} above.
\end{proof}

We conclude this section with the proof of Proposition \ref{full-line} reducing the full line case 
to the half line cases treated above.

\begin{proof}[Proof of Proposition \ref{full-line}]
Let $R^{\pm}=\ip{\delta_{\pm 1}}{\cdot} \delta_0+\ip{\delta_0}{\cdot} \delta_{\pm1}$ 
and let $H_{\pm}=H-R^{\pm}$. Applying the
resolvent formula to $\left(H-z \right)^{-1}$ with $z \in \bbC
\setminus \bbR$, one gets \beq \no \begin{split} \ip{\delta_{\pm
n}}{\left(H-z \right)^{-1}\delta_0} &= \ip{\delta_{\pm n}}{\left(H-z
\right)^{-1}\delta_0}  \\ & - \ip{\delta_{\pm n}}{\left(H_{\pm}-z
\right)^{-1}\delta_0}\ip{\delta_{\pm
1}}{\left(H-z \right)^{-1}\delta_0} \\
&-\ip{\delta_{\pm n}}{\left(H_{\pm}-z \right)^{-1}\delta_{\pm
1}}\ip{\delta_0}{\left(H-z \right)^{-1}\delta_0} \\
&=-\ip{\delta_{\pm
n}}{\left(H_{\pm}-z \right)^{-1}\delta_{\pm 1}}\ip{\delta_0}{\left(H-z \right)^{-1}\delta_0} \\
&=-\ip{\delta_n}{\left(H^{\pm}-z \right)^{-1}\delta_1}
\ip{\delta_0}{\left(H-z \right)^{-1}\delta_0}\end{split} \eeq for
any integer $ n \geq 1$, since $H_{\pm}$ are direct sums. Since
$\left | \ip{\delta_0}{\left(H-z \right)^{-1}\delta_0} \right |^2
\leq \frac{1}{(\Im z)^2}$ it follows that for $n>1$ \beq \no \left|
\ip{\delta_{\pm n}}{\left(H-E-i/T \right)^{-1}\delta_0} \right|^2
\leq T^2 \left| \ip{\delta_n}{\left(H^{\pm}-E-i/T
\right)^{-1}\delta_1} \right|^2 \eeq which immediately implies by
\eqref{Parseval} that \beq \no \int_0^\infty \left| \ip{\delta_{\pm
n}}{e^{-itH}\delta_0} \right|^2 e^{-2t/T}dt \leq T^2 \int_0^\infty
\left| \ip{\delta_n}{e^{-itH^{\pm}}\delta_1} \right|^2 e^{-2t/T}dt.
\eeq

Plugging the left hand side into the definition of
$P_{\delta_0}(q,T)$ and applying the above inequality to the
positive and negative $n$ separately we get
\eqref{full-line-from-half}.
\end{proof}


\section{An Analysis of the Fibonacci Hamiltonian}

In this final section we apply our method to get an upper bound for the dynamics 
of the Fibonacci Hamiltonian---$H_{\rm{F}}$. This is the whole-line operator with potential given by \eqref{fibonacci}. 
We shall describe its relevant properties below. For a more comprehensive review see \cite{d3}. We shall concentrate on the application of 
Theorem \ref{th-1.8}, but will also remark on the possibility of using 
Theorems \ref{Thouless-width} and \ref{clustering} to get weaker results with significantly less effort.

The unique spectral properties of $H_{\rm{F}}$ make it an ideal candidate for studying the relationship between spectral 
properties and dynamics. In particular, for all $\lambda$ the spectrum of $H_{\rm{F}}$ is a Cantor set and the spectral 
measure is always purely singular continuous. Anomalous transport has been indicated by various numerical works since the 
late 1980's (see e.g.\ \cite{ah,gkp,ah1}). 
In particular, work by Abe and Hiramoto \cite{ah,ah1} suggested that $\alpha_l^\pm$ and $\alpha_u^\pm$ behave like 
$\frac{1}{\log \lambda}$ as $\lambda \rightarrow \infty$ (recall $\lambda$ is the coupling constant in \eqref{fibonacci}).

In \cite{KKL} Killip, Kiselev and Last have proven both a lower bound and an upper bound on the slow moving part of the 
wave packet whose asymptotic behavior agrees with this prediction. As mentioned in the Introduction, an upper bound 
for the fast moving part of the wave packet was proven recently by Damanik and Tcheremchantsev in \cite{DT3} and in 
\cite{DT4} where they proved the same upper bound without time averaging. 
The Damanik-Tcheremchantsev bound reads: for $\lambda \geq 8$, 
$\alpha_u^+ \leq \frac{2 \log \eta}{\log \zeta(\lambda)}$, where 
$\eta=\frac{\sqrt{5}+1}{2}$ and $\zeta(\lambda)=\frac{\lambda-4+\sqrt{(\lambda-4)^2-12}}{2}$. We shall also need 
$r(\lambda)= 2\lambda+22$. 

Using Theorem \ref{th-1.8}, we shall show
\begin{theorem} \label{fib}
Let $\lambda >8$ and let 
$\alpha(\lambda)=\frac{3 \log r(\lambda)-\log\left(\zeta(\lambda) \eta \right)}{\log\left(r(\lambda) \eta \right)} 
\cdot \frac{2\log \eta}{\log(\zeta(\lambda))}$. Then
\beq \no
\alpha_u^+(\lambda) \leq \alpha(\lambda). 
\eeq
\end{theorem}
\begin{remark*}
We note that for $\lambda \geq 17$, $\alpha(\lambda)<1$ so this is a meaningful upper bound. 
Also, for $\lambda \geq 8$, 
$\frac{3 \log r(\lambda)-\log\left(\zeta(\lambda) \eta \right)}{\log\left(r(\lambda) \eta \right)} \leq 3$ and 
$\frac{3 \log r(\lambda)-\log\left(\zeta(\lambda) \eta \right)}{\log\left(r(\lambda) \eta \right)} \rightarrow 2$ as 
$\lambda \rightarrow \infty$.
\end{remark*}

Fix $\lambda > 8$. By Proposition \ref{full-line} and the symmetry of $H_{\rm{F}}$ 
($V_{\rm{Fib};-n}^\lambda=V_{\rm{Fib};n-1}^\lambda$ for 
$n\geq 2$), it is enough to consider the one-sided operator
$H_{\rm{F}}^+$ which is the restriction of $H_{\rm{F}}$ to $\bbN$. In order to apply Theorem \ref{th-1.8} we need to 
choose a sequence $\{q_\ell\}_{\ell=1}^\infty$. As in most works dealing with $H_{\rm{F}}$, we shall focus on the 
Fibonacci sequence: $F_\ell=F_{\ell-1}+F_{\ell-2}$ and $F_0=F_1=1$, and let $q_\ell \equiv F_\ell$. 
We recall that there exists a constant $C_\eta>0$ such that 
$C_\eta^{-1} \eta^\ell \leq q_\ell \leq C_\eta \eta^\ell$ so that, in particular, $q_\ell$ is exponentially growing. 

We need to show that $\mathfrak{E}_{q_\ell}$ is uniformly clustered by a sequence of interval families, 
$\{U_\ell\}_{\ell=1}^\infty$, that scales nicely. The relevant exponents will determine $\alpha(\lambda)$. Let 
$H_{\ell}^{\rm{per}}$ be the whole-line operator with potential $V_{q_\ell}^{\rm{per}}$ given by
\beq \no
V_{q_\ell;nq_\ell+j}^{\rm{per}}=V_{\rm{Fib};j}^\lambda  \quad 1 \leq j \leq q_\ell,
\eeq
namely, $V_{q_\ell}^{\rm{per}}$ is the $q_\ell$-periodic potential whose first $q_\ell$ entries coincide with those of 
$V_{\rm{Fib}}^\lambda$. As mentioned in the Introduction, the spectrum of 
$H_{\ell}^{\rm{per}}$, $\sigma_\ell$, is a set of intervals (bands). We will presently 
show that a natural cover for $\mathfrak{E}_{q_\ell}$ is provided by the bands in $\sigma_{m(\ell)}$ and 
$\sigma_{m(\ell)+1}$ for some $m(\ell)$ to be determined later.

Thus, we begin by considering $\sigma_{\ell}$. Following \cite{KKL}, we define a \emph{type A} band
as a band $I_{\ell}\subseteq\sigma_{\ell}$ such that $I_{\ell}\subseteq\sigma_{\ell-1}$
(so that $I_{\ell}\cap(\sigma_{\ell+1}\cup\sigma_{\ell-2})=\emptyset$). 
We define a \emph{type B} band as a band $I_{\ell}\subseteq\sigma_{\ell}$
such that $I_{\ell}\subseteq\sigma_{\ell-2}$ (and so $I_{\ell}\cap\sigma_{\ell-1}=\emptyset$). 
Letting $\sigma_{-1}=\mathbb{R}$ and $\sigma_{0}=[-2,2]$, and noting $\sigma_{1}=[\lambda-2,\lambda+2]$, 
we get that, for $\lambda>4$, $\sigma_{0}$ consists of one type A band and $\sigma_{1}$ consists of one type B band.

The structure of the spectrum of the Fibonacci Hamiltonian $H_{\rm{F}}$
can be deduced by using the following lemma (Lemma 5.3 in \cite{KKL}):

\begin{lemma} \label{lem-fibonacci_spec_building}
Assume $\lambda>4$. Then for any $\ell>0$ :
\begin{enumerate}
\item Every type A band $I_{\ell}\subseteq\sigma_{\ell}$ contains exactly
one type B band $I_{\ell+2}\subseteq\sigma_{\ell+2}$ and no other bands from
$\sigma_{\ell+1}$ or $\sigma_{\ell+2}$.
\item Every type B band $I_{\ell}\subseteq\sigma_{\ell}$ contains exactly
one type A band $I_{\ell+1}\subseteq\sigma_{\ell+1}$ and two type B bands
$I_{\ell+2,1}\subseteq\sigma_{\ell+2}$ and $I_{\ell+2,2}\subseteq\sigma_{\ell+2}$
located one on each side of $I_{\ell}$ . 
\end{enumerate}
\end{lemma}

Let $I_{k}^{B}$ be a type B band in $\sigma_{k}$. Using Lemma \ref{lem-fibonacci_spec_building}
one can construct, for $m>k$ a class $S_{k,m}^{B}$ of bands, belonging to $\sigma_{m}$, 
which are contained in $I_{k}^{B}$, i.e., if $I_{m}\subseteq\sigma_{m}$ and $I_{m}\in S_{k,m}^{B}$ then
$I_{m}\subseteq I_{k}^{B}$ . The same can be done for a type A band
$I_{k}^{A}\subseteq\sigma_{k}$, i.e., one can construct, by a repeated
use of Lemma \ref{lem-fibonacci_spec_building}, a class $S_{k,m}^{A}$
of bands in $\sigma_{m}$ such that if $I_{m}\in S_{k,m}^{A}$ then
$I_{m}\subseteq I_{k}^{A}$ (note that by Lemma \ref{lem-fibonacci_spec_building}
for $m=k+1$ we have $S_{k,k+1}^{A}=\emptyset$ ). 

Our analysis proceeds through the following

\begin{lemma} \label{lem-a_b_interval_chain}
Let $I_{k}^{B}\subseteq\sigma_{k}$ be a type B band. Then for $m\geq k\geq1$
we have $\#S_{k,m}^{B}=F_{m-k}.$ Let $I_{k}^{A}\subseteq\sigma_{k}$
be a type A band. Then for $k\geq0$ and $m\geq k+2$ we have $\#S_{k,m}^{A}=\#S_{k+2,m}^{B}=F_{m-k-2}$.
\end{lemma}

\begin{proof}
We notice that the procedure for the construction of the classes of
intervals $S_{k,m}^{B}$and $S_{k,m}^{A}$ is such that for fixed $l\in\mathbb{Z}$
with $k+l\geq 1$ we have $\#S_{k+l,m+l}^{B}=\#S_{k,m}^{B}$ and for
$k+l\geq 0$ we have $\#S_{k+l,m+l}^{A}=\#S_{k,m}^{A}$. By Lemma \ref{lem-fibonacci_spec_building}
we also have $\#S_{k,m}^{A}=\#S_{k+2,m}^{B}$ for $m\geq k+2$. Therefore,
$\#S_{k,m}^{B}=\#S_{1,m-k+1}^{B}$ and $\#S_{k,m}^{A}=\#S_{0,m-k}^{A}=\#S_{2,m-k}^{B}=\#S_{1,m-k-1}^{B}$.
The proof of Lemma \ref{lem-a_b_interval_chain} proceeds by induction.
Note first that $\#S_{m,m+1}^{B}=1=F_{1}$and $\#S_{m,m+2}^{B}=2=F_{2}$.
Assume that we know that $\#S_{m,m+l}^{B}=F_{l}$ for $l=0,1,\ldots,k\geq2$
and consider $\#S_{m,m+k+1}^{B}$. Let $I_{m}^{B}\subseteq\sigma_{m}$
be a type B band in $\sigma_{m}$. By Lemma \ref{lem-fibonacci_spec_building}
there is one type A band $I_{m+1}^{A}\subseteq\sigma_{m+1}$ with $I_{m+1}^{A}\subseteq I_{m}^{B}$
and two type B bands $I_{m+2,\, j}^{B}\subseteq\sigma_{m+2}$, $j=1,2$
with $I_{m+2,\, j}^{B}\subseteq I_{m}^{B}$. Therefore, for $k\geq1$
we have
\beq \no
\begin{split}
\#S_{m,m+k+1}^{B}&=\#S_{m+1,m+k+1}^{A}+2\#S_{m+2,m+k+1}^{B}=\#S_{m,m+k}^{A}+2\#S_{m,m+k-1}^{B}=\\
&=\#S_{m+2,m+k}^{B}+2\#S_{m,m+k-1}^{B}=\#S_{m,m+k-2}^{B}+2\#S_{m,m+k-1}^{B}= \\
&=F_{k-2}+2F_{k-1}=F_{k+1}.
\end{split}
\eeq
\end{proof}

We have obtained the following picture. The set $\sigma_m$ is made up of $F_m$ disjoint 
bands---$I_m^1,I_m^2,\ldots,I_m^{F_m}$. These bands are all disjoint to the type B bands of 
$\sigma_{m+1}$---$I_{m+1}^{B,1},I_{m+1}^{B,2},\ldots,I_{m+2}^{B,l(m)}$ while the type A bands of $\sigma_{m+1}$ 
are all contained in $\sigma_m$. Thus, by Lemma \ref{lem-a_b_interval_chain} the family 
\beq \no
\ti{U}_m \equiv=\{I_m^1,I_m^2,\ldots,I_m^{F_m},I_{m+1}^{B,1},I_{m+1}^{B,2},\ldots,I_{m+2}^{B,l(m)}\}
\eeq
is a cover for $\sigma_k$ for all $k \geq m$. Since $\mathfrak{E}_{q_\ell} \subseteq \sigma_\ell$ (more precisely, 
each element of $\mathfrak{E}_{q_\ell}$ is contained in a unique band of $\sigma_\ell$), we can take 
$U_\ell=\ti{U}_{m(\ell)}$ for a function $m(\ell) \leq \ell$ to be defined later.

We need two additional preliminary results.

\begin{lemma}[Proposition 5.2 in \cite{KKL}] \label{lem-KKL_a}
Assume that $\lambda>8$ and $k\geq 3$. Then, for every $E\in\sigma_{k}$ we have 
\beq \label{eq-5.1}
\left \vert \left( \calD^{q_k} \right)'(E) \right\vert\geq \zeta(\lambda)^{\frac{k}{2}},
\eeq
where $\zeta(\lambda)=\frac{\lambda-4+\sqrt{(\lambda-4)^2-12}}{2}$.  
\end{lemma}

\begin{lemma}[Equation 57 in \cite{DT}] \label{lem-KKL_b}
If $\lambda>4$ and $k \geq 1$. Then for every $E\in\sigma_k$ we have 
\beq \label{eq-5.2}
\left \vert \left( \calD^{q_k}\right)'(E) \right \vert\leq C(2\lambda+22)^k
\eeq
where $C$ is some positive constant. 
\end{lemma}

\begin{remark*}
It is a straightforward computation to see that Lemma \ref{lem-KKL_a} and Corollary \ref{cor-1.3} imply that 
$\alpha_l^+\leq \frac{6 \log \eta}{\log \zeta(\lambda)}$ for $H_{\rm{F}}^+$. 
The bounds we obtain for $\alpha_u^+$ are better so we do not elaborate on this point here.
\end{remark*}

We can now prove
\begin{proposition} \label{nice-scale-fib}
Let $\mu'(\lambda)=\frac{2 \log \left( 2\lambda +22\right)}{\log \zeta(\lambda)}$ and 
$\omega(\lambda)=\frac{2 \log \eta}{\log \zeta (\lambda)}$ $($recall $\eta=\frac{\sqrt{5}+1}{2})$. 
Then, for any $\mu(\lambda)>\mu'(\lambda)$, $\{\ti{U}_m\}_{m=1}^\infty$ scales nicely with exponents 
$\mu(\lambda)$ and $\omega(\lambda)$.
\end{proposition}
\begin{proof}
Fix $\nu >0$. Let $\mu(\lambda)= \mu'(\lambda)+\nu$. 
Now note that by the left hand side of \eqref{derivative-up-low-bound}, for any $I \in \ti{U}_m$ that satisfies 
$I \subseteq \sigma_m$
\beq \no
\left|\frac{4e}{(\calD^{q_m})'(\wti{E}_{q_m,I})}\right| \geq \left|I \right|
\eeq
where $\wti{E}_{q_m,I}$ is the unique zero of $\calD^{q_m}$ in $I$. In the same way, if $I \subseteq \ti{U}_m$ satisfying 
$I \subseteq \sigma_{m+1}$
\beq \no
\left|\frac{4e}{(\calD^{q_{m+1}})'(\wti{E}_{q_{m+1},I})}\right| \geq \left|I \right|.
\eeq
Thus, by Lemma \ref{lem-KKL_a}, it follows that for any $I \in \ti{U}_m$ 
\beq \label{eq-5.3}
\left| I \right| \leq 4e \zeta(\lambda)^{-m/2}\equiv \varepsilon_m.
\eeq

On the other hand, by the right hand side of \eqref{derivative-up-low-bound}, for $I \in \ti{U}_m$ satisfying 
$I \subseteq \sigma_m$, we have
\beq \no
\left|
I  \right| \geq \frac{\sqrt{5}+1}{\left|(\calD^{q_m})'(\wti{E}_{q_m,I}) \right|}
\eeq
and for $I \in \ti{U}_m$ satisfying 
$I \subseteq \sigma_{m+1}$ we get
\beq \no
\left|
I  \right| \geq \frac{\sqrt{5}+1}{\left|(\calD^{q_{m+1}})'(\wti{E}_{q_{m+1},I})\right|}.
\eeq
This, by Lemma \ref{lem-KKL_b}, implies that for all $I \in \ti{U}_m$
\beq \no
\left | I \right| \geq C (2\lambda +22)^{-(m+1)}= C (\varepsilon_m)^{\mu'(\lambda)+O\left(\frac{1}{m}\right)}
\geq \ti{C} \varepsilon_m^{\mu(\lambda)},
\eeq
for some constants $C>0$, $\ti{C}>0$. 

By Lemma \ref{lem-a_b_interval_chain}, for any $k>m$, any $I \in \ti{U}_m$ contains at most 
$F_{k+1-m}\leq C_\eta \eta^{k+1-m}$ elements of $\ti{U}_{k}$. Note also that
\beq \no
\eta^{k+1-m}=\eta\left( \zeta(\lambda)^{\frac{k-m}{2}}\right)^{\omega}=
\eta \left(\frac{\varepsilon_m}{\varepsilon_k}\right)^\omega.
\eeq

To sum up our present findings, we've shown that any element, $I$, of $\ti{U}_m$ satisfies
\beq \no
\varepsilon_m \geq \left| I \right| \geq \ti{C} \varepsilon_m^{\mu},
\eeq
for some constant $\ti{C}>0$, 
and that for any $k \geq m$ any element of $\ti{U}_k$ is contained in some element of $\ti{U}_m$ in such a way 
that there are no more than $C_\eta \eta \left(\frac{\varepsilon_m}{\varepsilon_k}\right)^\omega$ elements of 
$\ti{U}_k$ in any element of $\ti{U}_m$. 

To conclude the 
proof we only need to show that we may extend the sequence $\{\ti{U}_m\}_{m=1}^\infty$ of interval sets so that for 
any $\varepsilon$ we have a set $\ti{U}_\varepsilon$ of intervals in such a way that the family 
$\{ \ti{U}_\varepsilon\}_{0<\varepsilon<\varepsilon_1}$ satisfies these properties as well 
(perhaps with $C_\eta$ replaced by a different constant). 

But this is straightforward: Let $\varepsilon_m>\varepsilon>\varepsilon_{m+1}$ for some $m$. Now consider the elements of 
$\ti{U}_{m+1}$. Since each one is contained in an element of $\ti{U}_m$ which has 
length $\geq \ti{C} \varepsilon_m^{\mu(\lambda)}$, they may all be extended so that they are still inside the corresponding 
interval of $\ti{U}_m$ and their length is between $\varepsilon$ and $\ti{C}\varepsilon^\mu$ (they may intersect each other).
We take these extended intervals as the elements of $\ti{U}_\varepsilon$. Now it is easy to check that the family 
$\{ \ti{U}_\varepsilon\}_{0<\varepsilon<\varepsilon_1}$ satisfies the required properties, and this finishes the proof.
\end{proof}

\begin{proposition}\label{unif-cluster-fib}
Fix $\lambda>8$ $($note that $\omega(\lambda)\equiv \frac{2\log \eta}{\log \zeta(\lambda)}<1$  $)$ and choose $t$ 
so that $\omega(\lambda)<t<1$. Let $m(\ell)=[t \cdot \ell]$ and choose 
$\mu(\lambda)>\frac{2 \log \left( 2\lambda +22\right)}{\log \zeta(\lambda)}$.
Then there exists $L>0$ such that the sequence $\{\mathfrak{E}_{q_\ell}\}_{\ell=L}^\infty$ is uniformly 
clustered by $\ti{U}_{m(\ell)}$ with exponents $\{\alpha_\ell, \xi_\ell, \mu(\lambda)\}$, 
where $\alpha_\ell \equiv \frac{-\log (q_\ell)}{\log(\varepsilon_{m(\ell)})}$ $($with $\varepsilon_m$ as defined in 
\eqref{eq-5.3}$)$ and $\xi_\ell \equiv \frac{\log F_{\ell-m(\ell)-2}}{\log F_\ell}$, $($recall $F_m$ is the $m$'th 
Fibonacci number).
\end{proposition} 
\begin{proof}
Clearly, $q_\ell^{-1/\alpha_\ell}=\varepsilon_{m(\ell)}$ and $q_\ell^{\xi_\ell}=F_{\ell-m(\ell)-2}$ which, by Proposition 
\ref{nice-scale-fib} and Lemma \ref{lem-a_b_interval_chain}, say that $\mathfrak{E}_{q_\ell}$ is indeed 
$\{q_\ell^{-1/\alpha_\ell}, \xi_\ell\}$ clustered by $\ti{U}_{m(\ell)}$. 

Now, properties (i) and (ii) of Definition \ref{def-1.4} are obvious from Proposition \ref{nice-scale-fib} and so is 
property (iv) by Lemma \ref{lem-a_b_interval_chain}. We only have to check that $\delta<\xi_\ell<1-\delta$ and that 
$\delta<\alpha_\ell<1$ for some $\delta>0$. But $\lim_{\ell \rightarrow \infty} \xi_\ell=1-t$ and 
$\lim_{\ell \rightarrow \infty} \alpha_\ell= \frac{1}{t}\omega(\lambda)$ which, by the assumptions on $t$, implies 
this is true for $\ell$ sufficiently large.
\end{proof}

\begin{remark*}
For $\lambda$ sufficiently large ($\lambda>30$ suffices) one can choose $t<1/3$ above to show, using 
only Lemmas \ref{lem-a_b_interval_chain} and \ref{lem-KKL_a} together with Corollary \ref{cor-1.6}, that  
$\alpha_u^+ \leq \frac{6\log \eta}{\log \zeta(\lambda)}$. 
(Note $\lim_{\ell \rightarrow \infty} \xi_\ell=1-t>2/3$).
\end{remark*}

\begin{proof}[Proof of Theorem \ref{fib}]
As remarked above, it is enough to prove the upper bound for $\alpha_u^+$ associated with $H_{\rm{F}}^+$.
We shall show that for any $\delta>0$, $\alpha_u^+ \leq \alpha(\lambda)+\delta$. 

By Proposition \ref{unif-cluster-fib}, as long as $\omega(\lambda)<t<1$ and $m(\ell)=[t \cdot \ell]$, 
$\mathfrak{E}_{q_\ell}$ is uniformly clustered by $\ti{U}_{m(\ell)}$ with exponents $\{\alpha_\ell,\xi_\ell,\mu\}$
as defined above. By Proposition \ref{nice-scale-fib}, $\{\ti{U}_{m(\ell)}\}_{\ell=1}^\infty$ scales nicely with exponents
$\mu(\lambda)$ and $\omega(\lambda)$. Thus, in order to apply Corollary \ref{cor-1.10}, we only need to find $t$ such that 
\eqref{eq-1.5} holds for $\ell$ large enough. 

Since $\lim_{\ell \rightarrow \infty} \xi_\ell=1-t$ and 
$\lim_{\ell \rightarrow \infty} \alpha_\ell= \frac{1}{t}\omega(\lambda)$, this is guaranteed if
\beq \no
2\omega \left(\frac{\mu(\lambda)-1}{\mu(\lambda)-\omega(\lambda)} \right)<(1-t)\frac{1}{t}\omega(\lambda),
\eeq
namely, as long as
\beq \label{eq-5.4}
t<\frac{\mu(\lambda)-\omega(\lambda)}{3\mu(\lambda)-2-\omega(\lambda)}.
\eeq

Recall that $\mu(\lambda)=\mu_\nu(\lambda)=\mu'(\lambda)+\nu$ for some $\nu>0$. Thus, 
as long as 
\beq \no
t<\frac{\mu'(\lambda)-\omega(\lambda)}{3\mu'(\lambda)-2-\omega(\lambda)},
\eeq
inequality \eqref{eq-5.4} is guaranteed for some $\nu>0$ sufficiently small. We get that for such a $t$, the assumptions 
of Corollary \ref{cor-1.10} hold and we obtain $\alpha_u^+ \leq \frac{1}{t}\omega(\lambda)$, which implies that for any 
$\delta>0$,
\beq \no
\alpha_u^+ \leq \frac{3\mu'(\lambda)-2-\omega(\lambda)}{\mu'(\lambda)-\omega(\lambda)}\omega(\lambda)+\delta
\eeq
but elementary manipulations show this is the same as $\alpha_u^+ \leq \alpha(\lambda)+\delta$. 
\end{proof}

\end{document}